\documentclass[11pt]{article}

\usepackage{geometry}
\usepackage[pdftex]{graphicx}
\usepackage{amsfonts,amssymb,amsmath,amsthm}
\usepackage{epstopdf}
\usepackage{color}
\usepackage{hyperref,verbatim,float,stackrel,enumitem}
\usepackage[normalem]{ulem} 

\newcommand{\be}[1]{\begin{equation} #1 \end{equation}}
\newcommand{\bes}[1]{\begin{equation*} #1 \end{equation*}}

\newcommand{\eas}[1]{\begin{align*} #1 \end{align*}}

\def\R#1{(\ref{#1})}
\def\D{\,\mathrm{d}}
\def\I{\mathrm{i}}
\def\E{\mathrm{e}}

\def\imath{\mathrm{i}}
\newcommand{\abs}[1]{\left|#1\right|}

\newcommand*{\cL}{\mathcal{L}}

\newcommand*{\cR}{\mathcal{R}}

\newcommand*{\cW}{\mathcal{W}}

\newcommand*{\bbN}{\mathbb{N}}

\newcommand*{\bbR}{\mathbb{R}}

\newcommand*{\bbZ}{\mathbb{Z}}

\newcommand*{\nm}[1]{{\left\|#1\right \|}}

\numberwithin{equation}{section}

\newtheorem{theorem}{Theorem}
\newtheorem{remark}[theorem]{Remark}
\newtheorem{example}[theorem]{Example}
\newtheorem{lemma}[theorem]{Lemma}
\newtheorem{corollary}[theorem]{Corollary}
\newtheorem{proposition}[theorem]{Proposition}
\newtheorem{definition}[theorem]{Definition}
\newtheorem{problem}[theorem]{Problem}

\numberwithin{theorem}{section}

\newcommand{\sett}[1]{\ensuremath{\left \{ #1 \right \}}}
\newcommand{\norm}[1]{\ensuremath{\| #1 \| }}
\newcommand{\normo}[1]{\ensuremath{| #1 | }}
\newcommand{\Norm}[1]{\ensuremath{\big\lVert #1 \big\rVert }}
\newcommand{\mes}[1]{\textnormal{meas}(#1)}
\newcommand{\supp}{\text{supp}}
\newcommand{\func}{\psi}

\newcommand{\Rst}{\mathbb{R}}
\newcommand{\hRst}{\hat{\mathbb{R}}}
\newcommand{\hRdst}{\hat{\mathbb{R}}^d}
\newcommand{\Zst}{\mathbb{Z}}
\DeclareMathOperator*{\esssup}{ess\,sup}
\newcommand{\sinc}{\textnormal{sinc}}
\newcommand{\Rdst}{\bbR^d}

\newcommand{\head}[1]{\medskip \noindent $\bullet$ \emph{#1}}

\date{}

\title{Computing reconstructions from nonuniform Fourier samples: Universality of stability barriers and stable sampling rates}
\author{Ben Adcock\footnote{Department of Mathematics, Simon Fraser University, 8888 University Drive, Burnaby, BC V5A 
1S6, Canada; ben\_adcock@sfu.ca} \and Milana Gataric\footnote{
DPMMS, Faculty of Mathematics, University of Cambridge, Wilberforce Road, Cambridge, CB3 0WB, UK; 
m.gataric@maths.cam.ac.uk} \and Jos\'e Luis Romero 
\footnote{
Acoustics Research Institute, Austrian Academy of Sciences, Wohllebengasse 12-14, Vienna,
1040, Austria; \mbox{jlromero@kfs.oeaw.ac.at}}
}

\begin{document}

\maketitle

\paragraph{Keywords:} Nonuniform sampling, Generalized sampling, Stable recovery, Fourier frame bounds, Voronoi weights.

\begin{abstract}
We study the problem of recovering an unknown compactly-supported multivariate function from samples of its Fourier transform that are acquired nonuniformly, i.e.~not necessarily on a uniform Cartesian grid. Reconstruction problems of this kind arise in various imaging applications, where Fourier samples are taken along radial lines or spirals for example. 

Specifically, we consider finite-dimensional reconstructions, where a limited number of samples is available, and 
investigate the rate of convergence of such approximate solutions and their numerical stability. We show that the 
proportion of Fourier samples that allow for stable approximations of a given numerical accuracy 
is independent of the specific sampling geometry and is therefore universal for different sampling scenarios. This 
allows us to relate both sufficient and necessary conditions for different sampling setups and to exploit several 
results that were previously available only for very specific sampling geometries.

The results are obtained by developing: (i) a transference argument for different measures of the concentration of the Fourier transform and Fourier samples; (ii) frame  bounds valid up to the critical sampling density, which depend explicitly on the sampling set and the spectrum.

As an application, we identify sufficient and necessary conditions for stable and accurate reconstruction of algebraic polynomials or wavelet coefficients from nonuniform Fourier data.
\end{abstract}

\section{Introduction} 

\subsection{Non-uniform Fourier sampling}
Let $D \subseteq \Rdst$ be a compact domain and let $\hat{\bbR}^d$ denote the frequency domain. Consider the problem of reconstructing a function $f \in \cL^2(D)$ 
from samples of its Fourier transform 
\begin{equation}
\label{eq_def_ft}
\hat{f}(\xi) = \int_{\bbR^d} f(x) \E^{-\I 2\pi \xi \cdot x } \D x,\quad \xi \in \hat{\bbR}^d,
\end{equation}
taken on a countable subset $\Omega \subseteq \hRdst$, not necessarily a subset of an equidistant grid.   The abstract mathematical problem consists in establishing a \emph{frame 
inequality}
\begin{align}
\label{eq_frame}
A \norm{f}_2^2 \leq \sum_{\omega \in \Omega} \abs{\hat{f}(\omega)}^2 \leq B \norm{f}_2^2,\qquad f \in \cL^2(D),
\end{align}
with positive constants $A$ and $B$, which are called frame bounds. This problem is well-studied since it is equivalent to the one of 
reconstructing the bandlimited function $\hat{f}$ from its point samples $\sett{\hat{f}(\omega):\omega\in\Omega}$. The 
fundamental results of Duffin and Schaeffer, Kahane, Beurling and Landau relate the validity of the frame 
inequality \eqref{eq_frame} to the density of the set $\Omega$. In higher dimensions, the most effective criterion is 
due to 
Beurling \cite{BeurlingDiffOp, BeurlingVol2}. If $D$ is a centered symmetric convex body and $D^\circ$ is its polar set 
(see Section \ref{s:notation} for precise definitions), 
then $\Omega$ 
satisfies the sampling inequality \eqref{eq_frame} for some constants $A,B>0$
if its \emph{gap} with respect to $D^\circ$
\begin{align*}
\delta_{D^\circ}(\Omega) := \inf\left\{\delta>0: \Omega + \delta D^\circ = \hat{\bbR}^d \right\}
\end{align*}
is $<1/4$ and if $\Omega$ is relatively separated, i.e.~the number of points per unit volume is bounded. The value 
$1/4$ is critical in the sense that there are sets
$\Omega$ with $\delta_{D^\circ}(\Omega)=1/4$ that do not satisfy \eqref{eq_frame}. In one dimension and uniform (equidistant) sampling, condition $\delta_{D^{\circ}}(\Omega)=1/4$  coincides with the Nyquist sampling rate, which leads to Parseval's identity in \eqref{eq_frame}.

While Beurling's gap condition is very general and covers several situations of interest,
the computational aspects of the reconstruction problem, that are most relevant to applications, present a number of further challenges.

\begin{itemize}
\item[(i)] \emph{The frame bounds}. Classical sampling literature is concerned with the existence of 
frame bounds for \eqref{eq_frame}. However, computational problems  require 
explicit information on them, or at least a quantitative description on how they depend on the geometry of $D$ and 
$\Omega$. Typically, this information is practical in convergence and stability analysis of a given
reconstruction algorithm.

\item[(ii)] \emph{The geometry of $\Omega$}. Several applications demand very irregular sampling sets, with some 
groups of points being very close together and other far apart. This is the case in spiral sampling, for example, which 
is often used for fast acquisition of data in Magnetic Resonance Imaging, or in radial sampling, which is used 
whenever Radon data is acquired. Such clustering of sampling points is known to lead to a larger ratio $B/A$ of the 
frame bounds form \eqref{eq_frame}, which indicates worse stability of a reconstruction algorithm.

\item[(iii)] \emph{The approximation error and stability}. In order to compute the reconstruction $f|_{\Omega} \mapsto f$, we need to 
use a finite dimensional approximation and a finite set of sampling points. Computation of a stable and 
accurate approximation from finite data is of utmost importance for practical applications. 
The question that arises here is then, in which sense the finite 
dimensional setup reflects the full continuous problem. This problem is delicate and naive discretizations can 
lead to very poor reconstruction schemes (e.g. Gibbs phenomenon).  

\end{itemize}

With respect to (i), a common practice in the sampling literature is to establish sampling inequalities by means of 
oscillation estimates. This approach consists in constructing an explicit approximation of the unknown function 
$f$ using the samples $\left\{\hat{f}(\omega):\omega\in\Omega\right\}$ and in estimating the corresponding error using 
the controlled modulus of 
continuity of $\hat{f}$ \cite{GrochenigIrregular, GrochenigTrigonometric, AldroubiAverageSamp,
su07,AGH2DNUGS}. Such techniques yield explicit estimates of frame bounds, but higher-dimensional estimates of 
lower frame 
bounds are obtained only subject to density requirements that are considerably worse than the critical rate $1/4$. In 
this article, we provide explicit estimates of the lower frame bound subject to the sharp density condition 
$\delta_{D^\circ}(\Omega)<1/4$.

Challenge (ii) is normally addressed with the introduction of weights \cite{GrochenigIrregular, GrochenigTrigonometric,
Gabardo,AGH2DNUGS,1DNUGS,Rasche99,spyralScienceDirect, AldroubiGrochenigSIREV}. The most common choice are the 
\emph{Voronoi 
weights}
$\sett{\mu_\omega: \omega \in \Omega}$ which are the measures of the Voronoi regions associated with $\Omega$
(see Section \ref{s:notation}). Such a 
choice leads to upper frame bounds 
\begin{align}
\label{eq_W}
&W(\Omega) = \sup_{\substack{f\in \cL^2(\Rdst) \\ {\|f\|_2}=1}}  \sum_{\omega\in \Omega} \mu_{\omega}| 
\hat f (\omega) |^2
\end{align}
that are robust in the sense that adding more points does not increase the bound (see Section \ref{sec_bessel}). The 
analysis of the lower frame bound in the weighted setting is more challenging and part of this article is dedicated to 
that problem.

For the finite dimensional approximation problem (iii), one considers a reconstruction subspace $\cR \subseteq \cL^2(D)$ of dimension 
$N$ and a truncated set of points $\Omega_K := \Omega \cap B_K$, where $B_K$ is the Euclidean ball of radius $K$. Following \cite{1DNUGS}, we  let $f_N$ be the solution of the weighted finite-dimensional problem
\begin{align}\label{NUGS}
f_N = \underset{g \in \cR}{\operatorname{argmin}} \sum_{w \in \Omega_K} 
\mu_\omega \abs{\hat{f}(\omega)-\hat{g}(\omega)}^2,
\end{align}
and investigate the rate of convergence $\norm{f-f_N}_2$ for a particular sequence of subspaces 
$\cR=\cR_N$ (accuracy), and the $\cL^2$-norm of the map $f \mapsto f_N$ (stability). While $K$ represents a \emph{budget 
constraint} 
-- which portion of the infinite set $\Omega$ is involved in the actual sampling process -- $N$ represents an intended
\emph{resolution level} -- how accurate an approximation of $f$ we expect to recover from only 
a limited number of samples.   In particular, as discussed next, the relationship between $N$ and $K$ is 
critical for stability of numerical reconstruction. 

\subsubsection*{Stable sampling rates and stability barriers}

In \cite{1DNUGS} an abstract theory of finite-dimensional approximation of continuous sampling 
problems was introduced and the map $f \mapsto F_{\Omega_K,\cR}(f) := f_N$, where $f_N$ is defined as \R{NUGS}, was 
coined
\emph{non-uniform generalized sampling} (NUGS). The key quantity is the following \emph{concentration measure}
\begin{align}
\label{eq_V}
&V(\cR,\Omega,K) = \inf_{\substack{f\in\cR \\ \|f\|_2=1}}  \sum_{\omega\in \Omega_K} \mu_{\omega} | \hat f (\omega)|^2,
\end{align}
associated with $\Omega$. Stable reconstruction is possible if 
\be{\label{the_condition}
V(\cR,\Omega,K) >0
}
and in this case the convergence rate $\norm{f-f_N}_2$ is 
comparable to the best approximation rate obtainable in $\cR$. Indeed, we have the estimate \cite{1DNUGS}:
\begin{align}
\label{optim_const}
\|f- F_{\Omega_K,\cR}(f+g)\|_2 \leq \sqrt{\frac{W(\Omega_K)}{V(\cR,\Omega,K)}}
\left( \|f-Q_{\cR}(f)\|_2 +\|g\|_2 \right), \qquad f,g\in\cL^2(D),
\end{align}
where $Q_{\cR}$ denotes the orthogonal projection onto $\cR$,
and $W$ is given by \eqref{eq_W}. Therefore, for stable and accurate recovery from 
nonuniform Fourier samples it is \textit{sufficient} to provide conditions that ensure the ratio $\sqrt{W/V}$ 
is finite and small. As formalized in \cite{BAACHOptimality}, the condition \R{the_condition} is also \textit{necessary} for 
stable recovery by NUGS, or in fact, by any so-called contractive method.

The proportion of $N$ and $K$ sufficient for \R{the_condition} to hold
-- called the \emph{stable 
sampling rate} \cite{BAACHShannon} -- 
represents a trade-off between accuracy and stability and can be non-trivial even in simple cases. For example, when 
$D=[-1/2,1/2]$, $\Omega=\Zst$ and $\cR$ is the space of algebraic polynomials of degree $\leq N$, it 
was shown in 
\cite{grhr10} 
that $K \approx N^2$ gives a setup with stable sampling bounds
(see also \cite{BAACHAccRecov}, and \cite{AGH_PiecewiseSmooth2015} for nonuniform settings). Moreover, the proportion $K \gtrsim N^2$ was 
shown to be necessary for stability \cite{AdcockHansenShadrinStabilityFourier}. This latter phenomenon was coined a 
\emph{stability barrier}.

A number of recent articles establish stable sampling rates as well as stability barriers in several different 
contexts. For example, if $D=[-1/2,1/2]$, $\Omega=\Zst$ and $\cR$ is the $N$-dimensional space generated by compactly 
supported wavelets up to a certain scale, 
then  \cite{AHPWavelet} shows that $K \approx N$ suffices for stable recovery, while $K<N$ leads to exponential 
instability. The sufficiency part of these results was extended to two-dimensional wavelets in \cite{AHKM2DWavelets} and to two-dimensional 
shearlets in \cite{MaShearletsGS}. In addition, these results were extended to the context of nonuniform sampling in 
\cite{1DNUGS}, but only in the one-dimensional case.

In most cases, the derivation of stable sampling rates involve studying the following quantity
\begin{align}
\label{eq_Vstar}
&V_*(\cR,\Omega,K) = \sup_{\substack{f\in\cR \\ \|f\|_2=1}}
\sum_{\omega \in \Omega \setminus \Omega_K} \mu_{\omega} | \hat f (\omega)|^2,
\end{align}
which we call the \emph{residual} of the sampling set.
Estimates on $V_*(\cR,\Omega,K)$ can be used, in conjunction with a sampling theorem involving 
Fourier measurements on the whole $\hat{\bbR}^d$, to control the quantities $V(\cR,\Omega,K)$ \eqref{eq_V} and 
$W(\Omega_K)$ \eqref{eq_W} and thus obtain a stable numerical reconstruction form finite Fourier data.

In this article we obtain several results on stable sampling rates and stability barriers, covering 
different reconstruction spaces and nonuniform sampling sets in arbitrary dimensions with close to critical density. 
The results are obtained from a general method that allows us to remove the dependence of both $V(\cR,\Omega,K)$
and $V_*(\cR,\Omega,K)$ on the underlying sampling set $\Omega$. In this way, we are able to transfer stable sampling rates as well 
as stability barriers from uniform sampling to nonuniform sampling.

\subsection{Our contribution}
\label{sec_contrib}
In this article we contribute to (i), (ii) and (iii) in the following ways. (The corresponding proofs are given in 
Section \ref{sec_gran_finale}.)

\subsubsection*{Explicit estimates of lower frame bounds}
We derive an explicit estimate of the lower Fourier frame bound for general symmetric convex bodies $D$ and sampling sets 
$\Omega$ having gap $\delta_{D^\circ}(\Omega)$ arbitrarily close to the critical value $1/4$.
\begin{theorem}
\label{th_samp}
Let $D \subseteq \bbR^d$ be a centered symmetric convex body and let
$\Omega \subseteq \hat{\bbR}^d$ be a closed countable set such that $\delta_{D^\circ}=\delta_{D^\circ}(\Omega)<1/4$. 
Then
\begin{align*}
A \norm{f}_2^2\leq\sum_{\omega\in\Omega} \mu_\omega|\hat f(\omega)|^2,
\qquad f \in\cL^2(D),
\end{align*}
where
\begin{align}
\label{eq_frame_bound}
A &= \mes{D^\circ} \mes{D} \left(\frac{\delta_{D^{\circ}}\kappa^2}{6}\right)^d 
\cos\left(2\pi\delta_{D^\circ}(1+\kappa)^2\right)^2,
\\
\kappa &= \left( \frac{1}{\sqrt{4\delta_{D^{\circ}}}} - 1\right) \left( 1- \frac{1}{d+2} \right),
\end{align}
and $\mu_\omega$ are the Voronoi weights associated with the norm induced by $D^\circ$ -- see Section \ref{s:notation}. In particular, $A>0$ whenever $\delta_{D^\circ}<1/4$.
\end{theorem}
Theorem \ref{th_samp} requires no separation conditions and is valid even when the sampling set has high-density 
clusters. It improves on the multivariate estimates from \cite{GrochenigIrregular}, which were used to derive stability 
bounds for the Nonuniform Fast Fourier Transform in \cite{PottsTasche}. See also \cite{stta06,fegrst95}.

One important feature of the explicit bound in \eqref{eq_frame_bound} is that it 
shows that all sets $\Omega$ with $\delta_{D^\circ}(\Omega)
\leq \tfrac{1}{4}(1-\varepsilon)$ share a common lower frame bound, depending only on $\varepsilon$. This is essential for the 
applicability of our universality results below. 

\subsubsection*{Universality of stable sampling rates and stability barriers}
We prove that the stability / accuracy trade-off, expressed by stable sampling rates and stability barriers, is 
universal in Fourier sampling problems, in the sense that it is largely 
independent of the underlying geometry of the sampling set $\Omega$. We consider functions defined on a centered 
symmetric convex body $D\subseteq \bbR^d$. Following Theorem \ref{th_samp}, throughout this section, \emph{we consider 
the Voronoi weights associated with $\Omega$, with respect to the norm induced by $D^\circ$} -- see Section 
\ref{s:notation} for details. In particular the quantities $V(\cR,\Omega, K)$ and $V_*(\cR,\Omega, K)$ introduced
in \eqref{eq_V} and \eqref{eq_Vstar} are defined with respect to these weights.

\head{Transference of concentration and residual measures}.
Given a reconstruction subspace $\cR$, we consider the following quantities associated intrinsically with $\cR$:
\begin{align}
\label{eq_VR}
&V(\cR,K) = \inf_{\substack{f\in\cR \\ \|f\|_2=1}}  \int_{\abs{\xi} \leq K} \abs{\hat{f}(\xi)}^2 d\xi,
\\
\label{eq_VstarR}
&V_*(\cR,K) = \sup_{\substack{f\in\cR \\ \|f\|_2=1}}  \int_{\abs{\xi}>K} \abs{\hat{f}(\xi)}^2 d\xi.
\end{align}

We show that these quantities are essentially equivalent to the ones related to a sampling set $\Omega$,
given by \eqref{eq_V} and \eqref{eq_Vstar}. Specifically, 
we prove the following estimates.

\begin{theorem}
\label{th_1}
Let $D\subseteq \bbR^d$ be a centered symmetric convex body, and $\cR \subseteq \cL^2(D)$ a subspace.
Let $L>0$, $\Omega \subseteq \hat{\bbR}^d$ a closed countable set
such that $\delta_{B_1}(\Omega) \leq L$  and $\alpha \in (0,1)$. Then
\begin{align}
\label{eq_V1}
&V(\cR,\Omega, K) \leq C V(\cR,K+M) + C e^{-cM^\alpha}, \qquad K,M>0,
\\
\label{eq_Vstar1}
&V_*(\cR,\Omega, K) \leq C V_*(\cR,K-M) + C e^{-cM^\alpha}, \qquad K>M>0,
\end{align}
where $c,C$ are constants that only depend on $\alpha, L$ and $d$. 

In addition, if $\delta_{D^{\circ}}(\Omega) \leq 1/4(1-\varepsilon)$ for some $\varepsilon>0$
(i.e. the gap of $\Omega$ is below the critical value for the spectrum $D$), then
\begin{align}
\label{eq_V2}
& V(\cR,K) \leq C V(\cR,\Omega, K+M) + C e^{-cM^\alpha},
\\
\label{eq_Vstar2}
& V_*(\cR,K) \leq C V_*(\cR,\Omega, K-M) + C e^{-cM^\alpha},
\qquad K,M>0,
\end{align}
where $c,C$ are constants that depend only on $\varepsilon$, $\alpha$ and $d$.

(Here, the quantities $V(\cR,\Omega, K)$ and $V_*(\cR,\Omega, K)$ are defined by using the Voronoi weights
associated with $\Omega$ and the norm induced by $D^\circ$.)
\end{theorem}
We also provide a version of Theorem \ref{th_1} for the critical case $D=[-1/2,1/2]^d$
and $\Omega=\Zst^d$. The error decay is much milder in this case.
\begin{theorem}
\label{th_2}
Let $\cR \subseteq \cL^2([-1/2,1/2]^d)$ be a subspace. Then
\begin{align*}
& V(\cR,K) \leq C V(\cR,\Zst^d, M) + C \tfrac{K}{M},
\\
& V_*(\cR,K) \leq C V_*(\cR,\Zst^d, M) + C \tfrac{M}{K},
\qquad K,M>0,
\end{align*}
where $C$ is a constant that depends only on $d$. (Note that converse bounds are provided by
the first part of Theorem \ref{th_1}.)
\end{theorem}

Theorems \ref{th_1} and \ref{th_2} allow us to transfer stability results from one sampling set to another. 
While the estimates for residuals  are useful to transfer sufficient stable sampling rates, the 
ones on concentration measures are useful to transfer necessary conditions, i.e.\ stability barriers.

\head{Stable sampling rates}.
To be specific, we quantify decay rates with power laws. For $\alpha>0$, we say that a sequence of subspaces 
$\sett{\cR_N: N \geq 1}$ has \emph{residual decay
of order $\alpha$} if given $\theta>0$, there exists a constant $c_\theta>0$ such that 
$$
\sup_{N \geq 1} 
V_*(\cR_N,c_\theta N^\alpha) \leq \theta.
$$ 
The next result, that follows readily from Theorems \ref{th_samp} and 
\ref{th_1}, shows 
that the stable sampling rate in Fourier sampling is a notion intrinsically related 
to a reconstruction space, but does not depend on the underlying geometry of the sampling points. In particular, this 
improves on \cite[Thm.~3.3]{AGH2DNUGS}, by covering sampling sets with gap arbitrarily close to the critical value 
$1/4$, and \cite[Thm.~4.5]{1DNUGS}, by covering higher dimensions.

\begin{corollary}
\label{coro_univ_rate}
Let $D \subseteq \bbR^d$ be a centered symmetric convex body. Let $\sett{\cR_N: N \geq 1}$ be a sequence of subspaces 
of $\cL^2(D)$ with residual decay of order $\alpha$. Let $\varepsilon>0$. Then there exist constants $A,c>0$ 
that depend only on $\varepsilon$ and $D$, such that, 
for every closed countable set $\Omega \subseteq \hat{\bbR}^d$ with $\delta_{D^\circ}(\Omega) \leq 
\tfrac{1}{4}(1-\varepsilon)$,
the following stable sampling inequality holds:
\begin{align}
\label{eq_samp_univ}
A \norm{f}_2^2\leq\sum_{\omega\in \Omega \cap B_{c N^\alpha}} \mu_\omega|\hat f(\omega)|^2, 
\qquad f \in \cR_N,
N \geq 1.
\end{align}
\end{corollary}

This result provides a sufficient condition for stable recovery. In order to derive stable sampling rates for specific reconstruction spaces -- that is, a sufficient scaling of $N$ and $K$ for stable recovery in $\cR_N$ -- we need to combine Corollary \ref{coro_univ_rate} with residual estimates for concrete sequences of reconstruction subspaces. In several settings, these are available in the form of residual estimates for specific sampling sets. We say that $\sett{\cR_N: N \geq 1}$ has \emph{residual decay of order $\alpha$ with respect to the set $\Omega$} if given $\theta>0$, there  exists a constant $c_\theta>0$ such that 
$$\sup_{N \geq 1} V_*(\cR_N,\Omega,c_\theta N^\alpha) \leq \theta.$$ The following result shows that it is sufficient to estimate the order of residual decay on a specific sampling set.

\begin{corollary}
\label{coro_transfer}
Let $D \subseteq \bbR^d$ be a centered symmetric convex body and let $\sett{\cR_N: N \geq 1}$ be a sequence of 
subspaces of $\cL^2(D)$.
\begin{itemize}
\item[(a)] Let $\Omega \subseteq \hat{\bbR}^d$ be a closed countable set such that $\delta_{D^{\circ}}(\Omega) < 1/4$. 
If $\sett{\cR_N: N \geq 1}$ has residual decay of order $\alpha$ with respect to $\Omega$, then it has
residual decay of order $\alpha$.
\item[(b)] If $D=[-1/2,1/2]^d$ and $\sett{\cR_N: N \geq 1}$ has residual decay 
of order $\alpha$ with respect to $\mathbb{Z}^d$, then it has residual decay of order $\alpha$.
\end{itemize} 
\end{corollary}

\head{Stability barriers}. We say that $\alpha$ is a \emph{stability barrier for the sampling problem
associated with $\Omega$ and $\sett{\cR_N: N \geq 1}$} if for every $c>0$ and $\gamma >0$
the validity of the sampling bound 
$$\inf_{N \geq 1} V(\cR_N, \Omega, c N^\gamma) >0$$
implies that $\gamma \geq \alpha$.

\begin{corollary}
\label{coro_barrier}
Let $D \subseteq \bbR^d$ be a centered symmetric convex body and let $\sett{\cR_N: N \geq 1}$ be a sequence of 
subspaces of $\cL^2(D)$. Suppose that $\alpha$ is a stability barrier for the sampling problem associated with
a certain closed countable set $\Omega_0 \subseteq \hat{\bbR}^d$. Assume additionally that either
\begin{itemize}
\item $\delta_{D^{\circ}}(\Omega_0) < 1/4$ (i.e. the gap is below the critical value), or
\item $D=[-1/2,1/2]^d$ and $\Omega_0=\Zst^d$.
\end{itemize}
Then $\alpha$ is a stability barrier for the sampling problem associated with any 
closed countable set $\Omega \subseteq \hat{\bbR}^d$
with $\delta_{B_1}(\Omega) < +\infty$.
\end{corollary}

\subsubsection*{Concrete reconstruction results}

One important application of our results is the recovery of coefficients corresponding to orthogonal \textit{algebraic polynomials} from nonuniform Fourier samples. Polynomial reconstruction spaces are particularly suitable for recovery of (non-periodic) smooth functions, since these enjoy rapidly convergent approximations. For these reconstruction spaces, the results in \cite{grhr10, AdcockHansenShadrinStabilityFourier} establish stable 
sampling rates as well as a stability barrier in the case of sampling uniformly at critical density. By means of Theorems \ref{th_1}, \ref{th_2}, we extend 
these results to arbitrary sampling sets $\Omega$ and show that these stability conditions are independent of the 
particular geometry of $\Omega$. Such a conclusion does not follow from the methods in 
\cite{AdcockHansenShadrinStabilityFourier}, that rely heavily on a reformulation of Fourier sampling as a polynomial 
interpolation problem, which is only available for uniform sampling at critical density. 

Another important application of our results is the recovery of \textit{wavelet} coefficients from nonuniform Fourier samples. This case is particularly relevant to imaging applications, since images are known to be sparse in wavelets. In medical imaging, for example, it is vital to decrease the number of required measurements, and thanks to the aforementioned sparsity, regularization techniques such as compressive sensing can be instrumental. As argued in \cite{BAACHGSCS}, 
understanding the wavelet-specific
stable sampling rate and the stability barrier is a necessary first step, prior to embarking upon regularization methods. By means of Theorems \ref{th_1} and \ref{th_2}, we extend the results of \cite{AHPWavelet} and \cite{AHKM2DWavelets} to nonuniform sampling, as well as the results of \cite{1DNUGS} to higher dimensions. We show that a linear scaling between $K$ and $2^J$ is both sufficient and necessary for stable recovery of wavelet coefficients up to the wavelet scale $J$, from nonuniform Fourier measurements taken in the ball $B_K$. The practical implementation of such wavelet recovery was described in the recent work \cite{GPPractical}, and the theoretical results obtained here agree with and validate those observed in numerical experiments from \cite{GPPractical}.

\subsection{Technical overview}
A common technique in the computational sampling literature is to derive sampling inequalities by means of oscillation 
estimates (see e.g. \cite{GrochenigIrregular, fegrst95, alfe98, GrochenigTrigonometric, stta06, AldroubiAverageSamp, 
su07,nasuxi13,AGH2DNUGS}). These 
provide effective 
sampling 
bounds, but do not cover the complete range of sub-Nyquist gap densities. In contrast, Theorem \ref{th_samp} covers 
sets 
with density up to the critical value. In the unweighted case, the proof revisits Beurling's balayage techniques
\cite{BeurlingDiffOp}. More precisely, we follow a recent simple and powerful approach due to Olevskii and Ulanovskii 
\cite{OlevskiiUlanovskii} and quantify the main components of their argument. The case of weights is obtained 
afterwards by an argument from \cite{AGH2DNUGS}. 

For the problem of the universality of the stable sampling rate, we need to compare the effect of truncating 
frame expansions associated with different sampling sets. The challenge lies in the redundancy of these 
expansions, because setting some frame coefficients to zero has a spillover effect on the others. Indeed, 
when we identify a signal with its canonical frame coefficients, it turns out that
truncating 
a frame expansion is a Toeplitz-like operation: it sets some coefficients to zero and then projects the result onto 
the space of coefficients that are compatible with the restrictions imposed by redundancy. This perspective has been 
exploited in different contexts in \cite{ro11, ro12, doro14} and we use some technical insights from that work.

\subsection{Organization}
The rest of the article is organized as follows. Section \ref{s:notation} introduces the required definitions and 
notation.  Section \ref{sec_examples} presents the main applications of Theorems \ref{th_samp}, \ref{th_1} and 
\ref{th_2}. In Section \ref{sec_background} we provide some basic background on Fourier sampling. Section 
\ref{sec_real_work} contains our core technical contribution. We develop several estimates on truncation of Fourier 
expansions that are later used in Section \ref{sec_gran_finale} to prove Theorems \ref{th_samp}, \ref{th_1} and 
\ref{th_2}.

\paragraph{Acknowledgment.} The problems that led to this collaboration were proposed and discussed at the research 
cluster ``Computational Challenges in  Sparse and Redundant Representations'' held at Institute for Computational and 
Experimental Research in Mathematics (ICERM), Brown University, November 2014. The authors are very grateful to ICERM 
for its hospitality. B.~A.~acknowledges support from the NSF DMS grant 1318894, NSERC grant 611675 
and an Alfred P. Sloan Research Fellowship. M.~G.~acknowledges support from the
EPSRC Grant EP/N014588/1 for the EPSRC Centre for Mathematical and Statistical Analysis of 
Multimodal Clinical Imaging. J.~L.~R.~gratefully acknowledges support from a Marie 
Curie fellowship, within the 7th.~European Community Framework program, under grant 
PIIF-GA-2012-327063; from the Austrian Science Fund (FWF): P 29462 - N35; and
from the WWTF grant INSIGHT (MA16-053).

\section{Notation} \label{s:notation}
We introduce some definitions and fix the notation. The norm of a function $f \in \cL^2(\bbR^d)$ will be simply denoted 
$\| f \| := \norm{f}_2$. For a subset $D\subseteq\mathbb{R}^d$ we identify $\cL^2(D)$ with the subspace of 
$\cL^2(\bbR^d)$ formed by the functions supported on $D$.

A \emph{convex body} $D\subseteq\mathbb{R}^d$ is a compact convex set with non-empty interior. A convex body is called 
\emph{centered} if $0 \in \mathrm{int}(D)$ and \emph{symmetric} if $D=-D$. For a centered symmetric convex body $D$, 
the 
function $\abs{\cdot}_D:\mathbb{R}^d\rightarrow \mathbb{R}$ defined as 
\begin{align*}
|x|_D=\inf\left\{a>0 : x\in a D\right\}, \quad  x\in\bbR^d,
\end{align*}
is a norm on  $\mathbb{R}^d$. The \emph{polar set} of $D$ is
\begin{align*}
D^\circ=\left\{ z \in \mathbb{\hat{R}}^d : \forall x \in D,\ x\cdot z \leq1 \right\}.
\end{align*}
The Euclidean norm $\abs{\cdot}_2$ is simply denoted as $\abs{\cdot}$. Note that for the Euclidean norm we have 
$\abs{\cdot}=\abs{\cdot}_{B_1}=\abs{\cdot}_{B^{\circ}_1}$ where $B_1$ denotes the unit Euclidean ball.
For two non-negative functions $f,g$, we write $f \lesssim g$, if there exists a constant $C>0$ such that $f \leq C g$, and write $f \asymp g$ if $f \lesssim g$ and $g \lesssim f$.

\vspace{-1em}
\paragraph{Separation, density and bandwidth of sampling points:} Let $\Omega\subseteq\mathbb{\hat{R}}^d$ be a 
closed countable set, which we also refer to as a \textit{sampling set}. Given a norm $\abs{\cdot}_*$ on $\bbR^d$ and 
$\eta>0$, 
$\Omega$ is said to be $\eta$ \textit{separated} (with respect to $\abs{\cdot}_*$) if 
\begin{align*}
\forall \omega, \omega' \in \Omega , \quad \omega \neq \omega', \quad |\omega-\omega'|_* \geq \eta. 
\end{align*}
The set $\Omega$ is separated if it is $\eta$ separated for some $\eta>0$,
and it is \textit{relatively separated} if it is a finite union of 
separated sets. Equivalently, $\Omega$ is relatively separated if its \emph{covering number}
\begin{align}
\label{eq_nomega}
n_\Omega := \sup_{z \in \hat\bbR^d} \# \left( \Omega \cap \left( \{z\} + [0,1]^d \right) \right)
\end{align}
is finite. The \emph{gap} of $\Omega$ (with respect to $\abs{\cdot}_*$) is
\begin{align*}
\delta_*(\Omega)=\sup_{z \in \hat\bbR^d} \inf_{\omega\in \Omega}|\omega- z|_*, 
\end{align*}
and we say that $\Omega$ is \textit{$\delta_*$-dense}. If $\abs{\cdot}_{*} = \abs{\cdot}_{D}$ for a centered symmetric 
convex body $D$, we just write $\delta_{D}(\Omega)$. The number $\delta_{D}(\Omega)$ is the infimum of 
the numbers $\delta>0$ such that $\Omega+\delta D=\Rdst$.

The density condition corresponding to the gap $\delta_{D^{\circ}}(\Omega)=1/4$ is called the critical density
for sampling with spectrum $D$, which, as noted earlier, in one dimension and within 
the uniform setting coincides with the Nyquist sampling rate. 

\vspace{-1em}
\paragraph{Fourier frames:} Let
\bes{
e_{\omega}(x) = \E^{\I 2 \pi \omega \cdot x} \chi_{D}(x), \quad x \in \bbR^d, \ \omega\in\hat\bbR^d,
}
where $\chi_{D}$ is the indicator function of the set $D$. A countable family of functions 
$\{e_\omega\}_{\omega\in\Omega}\subseteq\cL^2(D)$  is said to be a \textit{Fourier frame} for 
$\cL^2(D)$ if there exist constants $A,B>0$ such that \eqref{eq_frame} holds. The constants $A$ and $B$ are called  
\textit{upper and lower frame bounds}, respectively. If \eqref{eq_frame} is replaced by
\bes{
A\|{f}\|^2 \leq \sum_{\omega\in\Omega}\mu_\omega|\hat{f}(\omega)|^2 \leq B \|{f}\|^2, \quad f\in\cL^2(D),
}
where $\mu_\omega > 0$ are some weights, then $\left\{\sqrt{\mu_\omega}e_\omega\right\}_{\omega\in\Omega}$  is called a 
\textit{weighted Fourier frame} for 
$\cL^2(D)$.
In this article, we use \textit{measures of Voronoi regions} as weights, which is a standard 
practice  in nonuniform sampling, see for example \cite{Rasche99, AldroubiGrochenigSIREV}. The Voronoi region at $\omega\in\Omega$, with 
respect to the norm $\abs{\cdot}_*$, is given by
\begin{align*}
V_\omega = \left \{  z \in \hat\bbR^d : \forall \lambda \in \Omega,\ \lambda\neq \omega ,\ |\omega- z|_{*} \leq 
|\lambda- 
z|_{*} \right\}.
\end{align*}
We will always assume that $\Omega$ is countable and closed. Under this assumption,
the Voronoi regions $\{V_\omega: \omega \in \Omega\}$ form an almost disjoint cover of $\hat\bbR^d$, 
i.e., 
$\hat\bbR^d = \bigcup_\omega V_\omega$ and 
$\text{meas}(V_\omega \cap V_{\omega'}) = 0$, if $\omega \not= \omega'$.
The Lebesgue measure of the Voronoi region $V_{\omega}$ is the Voronoi weight $\mu_\omega$:
\begin{align*}
\mu_\omega := \text{meas}\left(V_\omega\right)=\int_{\hat\bbR^d}\chi_{V_\omega}( x)\D  x. 
\end{align*}
Note that if $\Omega$ is separated, then $\mu_\omega \gtrsim 1$. 

We remark that the Voronoi weights associated with a certain set $\Omega$ depend on a choice of norm
for $\hat\bbR^d$. In the 
applications to sampling problems below, there is a distinguished convex body $D$ (called spectrum) and we will assume 
that the Voronoi weights are associated with the norm induced by the corresponding polar set $D^\circ$.

\section{Applications and examples}
\label{sec_examples}
We now present two concrete applications of our main results -- Theorems \ref{th_samp}, \ref{th_1}, and \ref{th_2}.  

\subsection{Reconstruction of polynomial coefficients}
We consider the reconstruction of polynomial coefficients of a compactly supported function by means of 
Fourier measurements. For uniform sampling at critical density, the exact sampling rate was derived in
\cite{grhr10, BAACHAccRecov} and the stability barrier in \cite{AdcockHansenShadrinStabilityFourier}.  We extend these 
results to nonuniform sampling and show that these phenomena are not particular of sampling geometry, but rather a 
feature of the chosen reconstruction space. Specifically, we have the following.
\label{ex_poly}
\begin{proposition}
\label{prop_fourier_poly}
Let $\mathcal{P}_N([-1,1])$ be the space of algebraic polynomials of degree $\leq N$ 
restricted to $[-1,1]$ and let $\Omega \subseteq \hat{\bbR}$ be
a closed countable set such that $\delta_{[-1,1]}(\Omega) < +\infty$.

\begin{itemize}
\item \emph{(Necessary sampling conditions)}.
Let $K_N \asymp N^\gamma$, for some $\gamma>0$. Suppose that 
for some $A>0$, the following stable sampling inequality holds all $N \gg 0$
\begin{align}
\label{eq_samp_pol}
\sum_{w \in \Omega \cap B_{K_N}} \mu_{\omega} \abs{\hat{f}(\omega)}^2 \geq A \norm{f}^2,
\qquad f \in \mathcal{P}_N([-1,1]).
\end{align}
Then $\gamma \geq 2$.

\item \emph{(Sufficient sampling conditions)}.
Suppose that $\delta_{[-1,1]}(\Omega) < 1/4$. Then, there exist $c>0$ and $A>0$ such that
\eqref{eq_samp_pol} holds for $K_N := c N^2$. Moreover, given $\varepsilon>0$,
$c$ and $A$ can be chosen uniformly for all sets $\Omega$ with $\delta_{[-1,1]}(\Omega) \leq \tfrac{1}{4} (1-\varepsilon)$.
\end{itemize}
\end{proposition}
\begin{proof}
The main result of \cite{AdcockHansenShadrinStabilityFourier} shows that $2$ is a stability barrier for the
sampling problem associated with $\tfrac{1}{2}\Zst$. Therefore, the necessity part follows from Corollary
\ref{coro_barrier} (after rescaling the problem by a factor of 2).

For the sufficiency, we use the fact that $\mathcal{P}_N$ has residual decay of order $2$ with respect to the sampling 
set $\tfrac{1}{2}\Zst$ \cite{grhr10, BAACHAccRecov}. By Corollary \ref{coro_transfer} we conclude that $\mathcal{P}_N$ 
has residual decay of order $2$ and therefore Corollary \ref{coro_univ_rate} gives the desired sampling bounds.
\end{proof}
\begin{remark}
The necessity part of Proposition \ref{prop_fourier_poly} does not follow directly from the methods in 
\cite{AdcockHansenShadrinStabilityFourier}, since these rely essentially on a reformulation of the Fourier sampling 
problem as a polynomial interpolation problem, and such a reformulation is only valid for uniform sampling at the 
critical rate.

Although we managed to extend the stability barrier to the irregular sampling setting, and beyond the Nyquist rate, 
the methods in this article do not allow us to recover the fine behavior of the condition number near the critical 
density. Such a description is however available for uniform sampling at the Nyquist rate \cite{grhr10, 
BAACHAccRecov, AdcockHansenShadrinStabilityFourier}.
\end{remark}

\subsection{Reconstruction of wavelet coefficients}
 
The results of \cite{AHPWavelet} show the following: $(a)$ the stability barrier for recovery of coefficients with respect 
to compactly supported wavelets from uniform Fourier samples acquired on $\bbZ$ is equal to $1$, and $(b)$ the space of 
compactly supported wavelets has residual decay of order $1$ with respect to the sampling set $\bbZ$, that is, the 
stable sampling rate is linear. Thus, in the same manner as above, by using Corollaries \ref{coro_transfer}, 
\ref{coro_univ_rate}
and \ref{coro_barrier}, we extend these stability conditions $(a)$ and $(b)$  to 
Fourier samples taken on a nonuniform sampling set $\Omega$. The stability barrier in \cite{AHPWavelet} uses the 
following lemma, that we shall exploit again.
\begin{lemma}
\label{lemma_con}
Let $\varepsilon \in (0,1)$. Then there exist constants $c_\varepsilon, C_\varepsilon >0$
such that for $n \in \mathbb{N}$, there exist a trigonometric polynomial
$m(\xi) = \sum_{k=0}^n c_k \E^{-\I 2\pi k \xi}$ such that $\norm{m}^2_{\cL^2([-1/2,1/2])}=1$
and $\norm{m}^2_{\cL^2([-\varepsilon,\varepsilon])} \leq C_\varepsilon e^{-c_\varepsilon n}$.
\end{lemma}
See \cite{MR1171553} for a construction of $m$ in terms of Chebyshev polynomials.

We now formulate precisely necessary and sufficient sampling conditions for recovery of two-dimensional boundary-corrected Daubechies wavelets from nonuniform Fourier samples. For the sufficiency part, we use the main result from
\cite{AHKM2DWavelets}.

First, we define the corresponding wavelet subspace in $\cL^2(D)$, $D=[-1/2,1/2]^2$, by following \cite{cohen1993wavelets} (see also \cite{AHKM2DWavelets,GPPractical}).  Let $\phi$ be a compactly supported Daubechies scaling function with $p$ vanishing moments and $\text{supp}(\phi)=[-p+1,p]$. For any $j\in\bbN$, $j\geq 1+\log p$, let
\bes{
\phi^{\text{b}}_{j,n} = \left \{ \begin{array}{ll} 2^{j/2} \phi (2^j \cdot - n), & -2^{j-1} + p \leq n < 2^{j-1} -p
\\ 
2^{j/2} \phi^{\text{left}}_{n} (2^j \cdot), & -2^{j-1}  \leq n < -2^{j-1}+p 
\\
2^{j/2} \phi^{\text{right}}_{2^j-n-1}(2^j(\cdot-1)), & 2^{j-1} -p \leq n < 2^{j-1},
\end{array} \right .
}
where $\phi^{\text{left}}_{n}$ and $\phi^{\text{right}}_{n}$ are left and right boundary-corrected scaling functions as 
defined in \cite{cohen1993wavelets}. Similarly, let $\psi^{\text{b}}_{j,n}$ denote the boundary-corrected wavelet 
function 
on $[-1/2,1/2]$. Let the two-dimensional scaling function be defined by tensor product as $\phi^{\text{b}}_{j,(n,m)} = 
\phi^{\text{b}}_{j,n}\otimes \phi^{\text{b}}_{j,m}$ and the wavelet function as
\bes{
\psi^{\text{b},k}_{j,(n,m)} = \left \{ \begin{array}{ll} \phi^{\text{b}}_{j,n}\otimes \psi^{\text{b}}_{j,m},  & k=1, 
\\ 
\psi^{\text{b}}_{j,n}\otimes \phi^{\text{b}}_{j,m}, & k=2,
\\
\psi^{\text{b}}_{j,n}\otimes \psi^{\text{b}}_{j,m}, & k=3,
\end{array} \right .
}
where $-2^{j-1}\leq n,m \leq 2^{j-1}-1$. We fix an integer $J_0\geq 1+\log p$ (base scale) and let
\eas{
\Phi_{J_0} &= \left\{ \phi^{\text{b}}_{j,(n,m)} : -2^{J_0-1}\leq n,m \leq 2^{J_0-1}-1 \right\}, \\
\Psi_{j} &= \left\{ \psi^{\text{b},k}_{j,(n,m)} : -2^{j-1} \leq n,m \leq 2^{j-1}-1, \ k=1,2,3 \right\},\ j\in\bbN,\ 
j\geq J_0.
}
The set 
\bes{
\cW := \Phi_{J_0} \cup \Big( \bigcup_{J_0\leq j}  \Psi_{j} \Big) 
}
forms an orthonormal basis for $\cL^2(D)$, $D=[-1/2,1/2]^2$, cf. \cite{cohen1993wavelets}. We consider the 
finite-dimensional subspace 
\be{\label{wave_space}
\cW_{N_J} = \text{span} \left\{ \varphi_n : \varphi_n\in \Phi_{J_0} \cup \Big( \bigcup_{J_0\leq j\leq J -1}  \Psi_{j} 
\Big) , \ n=1,\ldots,N_J \right\}
}
spanned by the $N_J = 2^{J}\times 2^J$ wavelets, up to the finest scale $J>J_0$. Since the boundary corrected 
Daubechies wavelets are associated with a multiresolution analysis \cite{cohen1993wavelets}, we also have $\cW_{N_J} = 
\text{span} (\Phi_{J})$.

We can now formulate the stable sampling result.

\begin{proposition}
\label{prop_fourier_wave} Let $D=[-1/2,1/2]^2$ and let $\Omega \subseteq \hat{\bbR}^2$ be a closed 
countable set.
Let $\cW_{N_J}$ be the space of boundary corrected Daubechies wavelets up to scale $J$ defined on 
$D$, as in $\R{wave_space}$. 

\begin{itemize}
\item \emph{(Necessary sampling conditions)}.
Let $K_J \asymp 2^{\gamma J}$, for some $\gamma>0$. Suppose that 
for some $A>0$, the following stable sampling inequality holds all $J \gg 0$
\begin{align}
\label{eq_samp_wav}
\sum_{w \in \Omega \cap B_{K_J}} \mu_{\omega} \abs{\hat{f}(\omega)}^2 \geq A \norm{f}^2,
\qquad f \in \cW_{N_J},
\end{align}
Then $\gamma \geq 1$.

\item \emph{(Sufficient sampling conditions)}.
Suppose that $\delta_{D^{\circ}}(\Omega) < 1/4$. Then, there exist $c>0$ and $A>0$ such 
that \eqref{eq_samp_wav} holds for $K_J := c 2^J$. Moreover, given $\varepsilon>0$,
$c$ and $A$ can be chosen uniformly for all sets $\Omega$ with $\delta_{D^{\circ}}(\Omega) \leq \tfrac{1}{4} 
(1-\varepsilon)$.
\end{itemize}
\end{proposition}
\begin{proof}
Let us prove the necessity. We want to show that $1$ is a stability barrier for the sampling 
problem associated with $\Omega$ and $\cW_{N_J}, J \gg 0$. By Corollary \ref{coro_barrier}, it 
suffices to consider one specific sampling set; we assume that $\Omega=\mathbb{Z}^2$. We prove that
$1$ is a stability barrier by showing the contrapositive of the condition in the definition;
let $\gamma \in (0,1)$, and let us show that, for all $c>0$, 
$V(\cW_{N_J}, \mathbb{Z}^2, c 2^{\gamma J})\longrightarrow 0$,
as $J \longrightarrow +\infty$.

Let
\[H_p := \sett{(k_1, k_2) \in \mathbb{Z}^2: -2^{J-1} + p \leq k_i  < 2^{J-1} -p}\]
be the indices corresponding to the interior scaling functions.
Note that $\# H_p \asymp 2^{J}$.
Let
\begin{align*}
f(x) = \sum_{k_1,k_2 \in H_p} a_{k_1,k_2} 2^{J} \phi(2^J x_1-k_1)\phi(2^J x_2-k_2),
\end{align*}
with $\norm{a}_2=1$. Then $f \in \cW_{N_J}$.
We now estimate
\begin{align*}
V(\cW_{N_J}, K) &\leq \int_{\abs{\xi} \leq K} \abs{\hat{f}(\xi)}^2 d\xi
\leq
\int_{\abs{\xi}\leq K}
\abs{\sum_{k_1,k_2 \in H_p} a_{k_1,k_2} 2^{-J} 
\E^{-\I 2\pi 2^{-J} k \xi}
\hat{\phi}(2^{-J}\xi_1)\hat{\phi}(2^{-J}\xi_2)}^2 
d\xi 
\\
&\lesssim \int_{\abs{\xi}\leq K} 
\abs{\sum_{k_1,k_2 \in H_p} a_{k_1,k_2} 2^{-J} 
\E^{-\I 2\pi 2^{-J} k \xi}}^2 d\xi
=\int_{\abs{\xi}\leq 2^{-J}K} 
\abs{\sum_{k_1,k_2 \in H_p} a_{k_1,k_2} \E^{-\I 2\pi k \xi}}^2
d\xi.
\end{align*}
Now we let $K:= c 2^{\gamma J}$ with $c>0$. Using that $\gamma < 1$ and
that $\# H_p \asymp 2^{J}$, and letting $a$ be the tensor product of adequate shifts of the 
sequences provided
by Lemma \ref{lemma_con}, we find that $V(\cW_{N_J}, c 2^{\gamma J}) \longrightarrow 0$, for every $c>0$.
Using Theorem \ref{th_1}, we conclude that $V(\cW_{N_J}, \mathbb{Z}^2, c 2^{\gamma J})\longrightarrow 0$, for every 
$c>0$, as desired.

For the sufficiency, we invoke \cite[Theorem 4.3]{AHKM2DWavelets} that 
gives the following stable sampling estimate with respect to $\bbZ^2$:
for every $\theta \in (0,1)$, there exists $c_{\theta}>0$ such that
$V(\cW_{N_J},\bbZ^2,c_{\theta}2^J) \geq 1-\theta$. By Parseval's identity:
$V_*(\cW_{N_J},\bbZ^2,c_{\theta}2^J) = 1 -V(\cW_{N_J},\bbZ^2,c_{\theta}2^J)$ and therefore we conclude that
that $\{\cW_{N_J}: J\geq 1\}$ has residual decay of order $1$ with respect 
to the sampling set $\bbZ^2$. We now invoke Corollaries \ref{coro_transfer} and  \ref{coro_univ_rate}, 
and the conclusion follows.
\end{proof}

\section{Sampling background}
\label{sec_background}
We now collect some background and auxiliary results on nonuniform sampling.

\subsection{Balayage of delta measures}
The following was proved by Beurling \cite{BeurlingDiffOp} in a slightly weaker form, who mentioned without proof 
a possible refinement. The stronger version is proved by Olevskii and Ulanovskii in \cite[Theorem 
4.1]{OlevskiiUlanovskii}; see also 
the work of Benedetto and Wu \cite{BenedettoSpiral}.

\begin{proposition}
\label{prop_linf}
Let $D \subseteq \bbR^d$ be a centered symmetric convex body and let
$\Omega \subseteq \hat{\bbR}^d$ be such that 
$\delta_{D^\circ}(\Omega)<1/4$. Then for every distribution $f$ with support in $D$:
\begin{align*}
\cos\left(2\pi \delta_{D^\circ}(\Omega)\right)
\norm{\hat{f}}_\infty \leq \sup_{\omega\in\Omega}|\hat f(\omega)|,
\end{align*}
where $\nm{\cdot}_{\infty}$ denotes the $\cL^\infty$ norm.
\end{proposition}

In the setting of Proposition \ref{prop_linf}, the set $\Omega$ also gives rise to a Fourier frame for $\cL^2(D)$. Moreover,
Beurling's linear balayage method provides a way to quantify the corresponding frame bounds \cite{BeurlingDiffOp, 
BeurlingVol2}. We derive such explicit bounds in Corollary \ref{coro_expl}.

\subsection{Bessel bounds}
\label{sec_bessel}
We will use the following Bessel bounds. For proofs see for example \cite{GrochenigIrregular,Young}. 
\begin{lemma}
\label{lemma_bessel}
Let $D \subseteq \Rst^d$ be compact and let $L>0$. If $\Omega \subseteq \hat{\bbR}^d$ is a closed countable set such 
that $\delta_{B_1}(\Omega) \leq L$, then
the following Bessel condition holds
\begin{align}
\label{eq_bessel_1}
\left(\sum_{\omega \in \Omega} \mu_\omega \left|\hat{f}(\omega)\right|^2 \right)^{1/2} \leq C_{d,L} \norm{f},
\qquad
f \in \cL^2(D),
\end{align}
where $C_{d,L}$ depends only on the dimension $d$ and the bound on the gap $L$, and $\sett{\mu_\omega: \omega \in 
\Omega}$ are the Voronoi weights with respect to the Euclidean norm. A similar statement holds for the Voronoi weights 
associated with any other norm, with possibly different constants.

In contrast, if $\Omega\subseteq \hat{\bbR}^d$ is relatively separated, we have
\begin{align}
\label{eq_bessel_2}
\left(\sum_{\omega \in \Omega} \left|\hat{f}(\omega)\right|^2\right)^{1/2} \leq C_D
n_\Omega \norm{f},
\qquad
f \in \cL^2(D),
\end{align}
where $C_D$ only depends on $D$ and $n_\Omega$ is the covering number of $\Omega$, cf. \eqref{eq_nomega}.
\end{lemma}

\begin{remark}\label{r:upper_bd} The constants in Lemma \ref{lemma_bessel}
can be described explicitly in a number of situations. For \eqref{eq_bessel_1}, \cite{AGH2DNUGS} 
gives the following estimates that improves on a previous result from \cite{GrochenigIrregular}.
If $D$ is a centered symmetric convex body, and $\Omega \subseteq \hRdst$ has gap $\delta_{D^{\circ}}$, then
\begin{align*}
\left(\sum_{\omega \in \Omega} \mu_\omega \left|\hat{f}(\omega)\right|^2 \right)^{1/2} \leq \exp\left(2\pi \delta_{D^{\circ}} 
c^{\circ} m_{D}\right) \norm{f},
\qquad
f \in \cL^2(D),
\end{align*}
where $\sett{\mu_\omega: \omega \in \Omega}$ are the Voronoi weights with respect to the 
$\abs{\cdot}_{D^{\circ}}$-norm,  $m_{D}=\max_{x\in D}\abs{x}$ and $c^{\circ}$ is the smallest constant such that 
$\abs{\cdot}\leq c^{\circ}\abs{\cdot}_{D^{\circ}}$. 
\end{remark}

\subsection{Weights and subsets} 
The next result allow us to derive weighted sampling inequalities by selecting adequate separated subsets of a given 
sampling set. Other estimates in this spirit can be found in \cite[Thm.~1.1]{AGH2DNUGS}.

\begin{lemma}
\label{lemma_add_weights}
Let $D \subseteq \bbR^d$ be a centered symmetric convex body.
Let $\Omega \subseteq \hRdst$ be a closed countable set with gap $\delta_{D^\circ}(\Omega)$, and let 
$\rho,\eta>0$ be such that $\rho<\eta/2$. 
Let $f:\hRdst \to \Rst$ be a continuous function. Then 
there exists a set 
$\bar\Omega=\bar\Omega (f) \subseteq \hRdst$, which is $(\eta-2\rho)$-separated with respect to the 
$\abs{\cdot}_{D^{\circ}}$-norm, has gap
$\delta_{D^\circ}(\bar\Omega) \leq \delta_{D^\circ}(\Omega)+\eta+\rho$, and is such that for any set $Y \subseteq \hRdst$
\begin{align}
\label{eq_lemma_add_weights}
\sum_{\omega \in \Omega \cap ( Y +2\rho D^{\circ})} \mu_{\omega}\abs{f(\omega)}^2  \geq 
\textnormal{meas}\left(\frac{\rho}{2} D^\circ\right)
\sum_{\omega \in \bar\Omega\cap Y } \abs{f(\omega)}^2,
\end{align}
where $\sett{\mu_\omega: \omega \in \Omega}$ are the Voronoi weights associated with $\Omega$, with respect to the 
$\abs{\cdot}_{D^{\circ}}$-norm.
\end{lemma}
\begin{proof} 
Let $\tilde \Omega$ be a subset of $\Omega$ with separation at least $\eta>0$ with respect to the 
$\abs{\cdot}_{D^{\circ}}$-norm, and maximal with respect to this property. (The existence of such a subset follows from Zorn's lemma.) By maximality, $\Omega \subseteq \tilde \Omega + \eta D^{\circ}$ and therefore $\delta_{D^\circ}(\tilde\Omega)\leq\delta_{D^\circ}(\Omega)+\eta$. Since $\rho<\eta/2$, the sets 
$\tilde\omega+\rho D^{\circ}$, $\tilde\omega\in\tilde\Omega$, are disjoint. Since
$\bigcup_{\tilde\omega\in\tilde\Omega\cap (Y+\rho D^{\circ})} (\tilde\omega+\rho D^{\circ}) \subseteq  Y+2\rho 
D^{\circ}$, we conclude that
\begin{equation}
\label{eq_1} 
\begin{aligned}
\sum_{\tilde \omega\in\tilde\Omega 
\cap (Y+\rho D^{\circ})} \sum_{\omega \in\Omega\cap (\tilde\omega+\rho D^{\circ})} \mu_{\omega} \abs{f(\omega)}^2 &= \sum_{\omega\in\Omega\cap\left(\bigcup_{\tilde\omega\in\tilde\Omega\cap (Y+\rho D^{\circ})} (\tilde\omega+\rho D^{\circ})\right)}  \mu_{\omega} \abs{f(\omega)}^2\\
&\leq \sum_{\omega\in\Omega\cap ( Y +2\rho D^{\circ}) } \mu_{\omega} \abs{f(\omega)}^2.
\end{aligned}
\end{equation}
Let us choose a set $\bar\Omega=\bar\Omega(f)=\{z_{\tilde\omega} : \tilde\omega\in\tilde\Omega \}$, where for each 
$\tilde\omega\in\tilde\Omega$, the point $z_{\tilde\omega}$ is taken form the set $\tilde\omega+\rho D^{\circ}$ such that 
$\abs{f(z_{\tilde\omega})} \leq \abs{f(y)}$ for all $y\in \tilde\omega+\rho D^{\circ}$. By construction, the set 
$\bar\Omega$ is $(\eta-2\rho)$-separated with respect to the $\abs{\cdot}_{D^{\circ}}$-norm and has gap
$\delta_{D^\circ}(\bar\Omega) \leq \delta_{D^\circ}(\tilde\Omega)+\rho
\leq \delta_{D^\circ}(\Omega)+\eta + \rho$. 

We note that for each $\tilde\omega \in \tilde\Omega$, $\bigcup_{\omega \in\Omega\cap (\tilde\omega+\rho D^{\circ})} 
V_{\omega} \supseteq  \tilde\omega+ \frac{\rho}{2} D^\circ$.
(Indeed, if $x \in \tilde\omega+ \frac{\rho}{2} D^\circ$ and
$x \in V_{\omega}$, then, by definition,
$\abs{x-\omega}_{D^{\circ}} \leq \abs{x-\tilde\omega}_{D^{\circ}} \leq \frac{\rho}{2}$
and consequently $\abs{\tilde\omega-\omega}_{D^{\circ}} \leq \rho$.)
As a consequence,
\begin{equation}
\label{eq_2}
\begin{aligned}
\sum_{\tilde \omega\in\tilde\Omega 
\cap (Y+\rho D^{\circ})} \sum_{\omega \in\Omega\cap (\tilde\omega+\rho D^{\circ})} \mu_{\omega} \abs{f(\omega)}^2  &\geq  \sum_{\tilde\omega\in\tilde\Omega \cap (Y+\rho D^{\circ})} \abs{f(z_{\tilde\omega})}^2 \sum_{\omega \in\Omega\cap (\tilde\omega+\rho D^{\circ})} \mu_{\omega} \\
&\geq  \textnormal{meas}\left(\frac{\rho}{2} D^\circ\right)
\sum_{\omega \in \bar\Omega\cap  Y } \abs{f(\omega)}^2.
\end{aligned}
\end{equation}
The desired conclusion follows by combining \eqref{eq_1} and \eqref{eq_2}.
\end{proof}

\section{Main technical estimates}
\label{sec_real_work}
In this section we prove our main technical results on comparing truncations of the Fourier transform and different 
sampling sets. As explained in Section \ref{sec_contrib}, we rely on methods from \cite{AGH2DNUGS, 
OlevskiiUlanovskii,ro11, 
ro12, doro14}.

\subsection{Domination of the residual and the partial sum for nonuniform sampling}
\begin{proposition}\label{prop_residual}
Let $D \subseteq \bbR^d$ be compact and let $L>0$ and
$\alpha \in (0,1)$. Then
there exist constants $c=c_{\alpha,L}, C=C_{\alpha,L} >0$, such that for
all measurable sets $Y \subseteq \hat{\bbR}^d$ and
all closed countable sets $\Omega \subseteq \hat{\bbR}^d$ with $\delta_{B_1}(\Omega) \leq L$:
\begin{align}
\label{eq_aaa}
\left(
\sum_{\omega \in \Omega \cap Y}
\mu_\omega \left|\hat{f}(\omega)\right|^2 
\right)^{1/2}
\leq C
\left(
\int_{Y + B_R} \left|\hat{f}(\xi)\right|^2 d\xi
\right)^{1/2}
+ C e^{-c\abs{R}^\alpha} \norm{f},
\qquad
f \in \cL^2(D),\ R>0.
\end{align}
Here, $\sett{\mu_\omega: \omega \in \Omega}$ are the Voronoi weights with respect to the Euclidean norm.
A similar statement holds for the Voronoi weights associated with any other norm, with possibly different constants.
\end{proposition}
\begin{proof}
For notational simplicity we only discuss the case of the Euclidean norm; the same arguments apply to arbitrary 
norms.
If $0<R \leq L$, we use the Bessel bounds in Lemma \ref{lemma_bessel} to obtain
\begin{align*}
\sum_{\omega \in \Omega \cap Y}
\mu_\omega \left|\hat{f}(\omega)\right|^2 &\leq 
\sum_{\omega \in \Omega}
\mu_\omega \left|\hat{f}(\omega)\right|^2 \leq C_L \norm{f}^2
= C_L e^{c\abs{R}^\alpha} e^{-c\abs{R}^\alpha} \norm{f}^2
\\
& \leq 
\left[C_L e^{c\abs{L}^\alpha}\right] e^{-c\abs{R}^\alpha} \norm{f}^2.
\end{align*}
Hence, \eqref{eq_aaa} follows.
Let us assume that $R > L$ and
let $\func: \bbR^d \to [0,+\infty)$ be a function such that
$\func \equiv 1$ on $D$, $\supp(\func)$ is compact and
\begin{align}
\label{eq_decay1}
\abs{\hat{\func}(\xi)} \leq C e^{-c \abs{\xi}^\alpha},
\qquad \xi \in \hat{\bbR}^d,
\end{align}
for some constants $C,c$. (The existence of such a function follows from
\cite[Theorem 1.3.5]{ho83-1}. Alternatively, given a function $\func$ with the desired properties,
except possibly the fact that $\func \equiv 1$ on $D$, this latter property can be achieved by
convolving $\func$ with the characteristic function of a big cube.)

Let $\Omega \subseteq \hat{\bbR}^d$ be a set with $\delta_{D^{\circ}}(\Omega) \leq L$ and let $f \in \cL^2(D)$. Let us 
consider
\begin{align*}
&S\hat{f} := \sum_{\omega \in \Omega \cap Y} \hat{f}(w)
\cdot \chi_{V_\omega},
\\
&S^R\hat{f} :=
\sum_{\omega \in \Omega \cap Y}
\left(\hat{f} \cdot \chi_{Y + B_R}
\right) * \hat{\func} (w) \cdot \chi_{V_\omega}.
\end{align*}
Since $\supp(\func)$ is compact, the function
$\left(\hat{f} \cdot \chi_{Y + B_R}
\right) * \hat{\func}$ is bandlimited. Using this fact 
and Lemma \ref{lemma_bessel} we estimate
\begin{align*}
\left(
\sum_{\omega \in \Omega \cap Y}
\left|\hat{f}(\omega)\right|^2 \mu_\omega
\right)^{1/2}
&= \norm{S\hat{f}}_2
\leq \norm{S^R \hat{f}}_2 + \norm{(S-S^R)\hat{f}}_2
\\
&= \left(
\sum_{\omega \in \Omega \cap Y}
\abs{
\left(
\hat{f} \cdot \chi_{Y + B_R} \right)*\hat{\func}
(w)}^2 \mu_\omega
\right)^{1/2}
+ \norm{(S-S^R)\hat{f}}_2
\\
&\lesssim \norm{\left(\hat{f} \cdot \chi_{Y + B_R} \right)*\hat{\func}}_2
+ \norm{(S-S^R)\hat{f}}_2
\\
&\leq \norm{\func}_\infty \norm{\hat{f} \cdot \chi_{Y + B_R}}_2 
+ \norm{(S-S^R)\hat{f}}_2
\\
&\lesssim
\left(
\int_{Y + B_R} \left|\hat{f}(\xi)\right|^2 d\xi
\right)^{1/2}+\norm{(S-S^R)\hat{f}}_2.
\end{align*}
Therefore, it suffices to show that
$\norm{(S-S^R)\hat{f}}_2 \lesssim e^{-c\abs{R}^\alpha} \norm{f}_2$. 
Since $\supp(f) \subseteq D$, and $\func \equiv 1$ on $D$,
it follows that $\hat{f}=\hat{f}*\hat{\func}$. Hence,
we can write
\begin{align*}
(S-S^R)\hat{f}(\xi)&=
\sum_{\omega \in \Omega \cap Y}
\left(\hat{f} * \hat{\func}(w) - 
\left(\hat{f} \cdot \chi_{Y + B_R}
\right) * \hat{\func} (w)\right) \cdot \chi_{V_\omega}(\xi)
\\
&=
\sum_{\omega \in \Omega \cap Y}
\left(\hat{f} \cdot \chi_{\hat{\bbR}^d \setminus (Y + B_R)}
\right) * \hat{\func} (w) \cdot \chi_{V_\omega}(\xi)
\\
&=\int_{\hat{\bbR}^d}
\hat{f}(\eta) K^R(\eta,\xi) d\eta,
\end{align*}
where
\begin{align*}
K^R(\eta,\xi) := 
\sum_{\omega \in \Omega \cap Y}
\chi_{\hat{\bbR}^d \setminus (Y + B_R)}(\eta)
\hat{\func}(\omega-\eta) \cdot \chi_{V_\omega}(\xi).
\end{align*}
Using \eqref{eq_decay1} we obtain the bound
\begin{align}
\label{eq_kernel1}
\abs{K^R(\eta,\xi)} 
\lesssim \sum_{\omega \in \Omega \cap Y}
\chi_{\hat{\bbR}^d \setminus (Y + B_R)}(\eta)
e^{-c \abs{\omega-\eta}^\alpha} \cdot \chi_{V_\omega}(\xi),
\qquad
\eta,\xi \in \hat{\bbR}^d.
\end{align}

We use Schur's lemma to bound the integral operator with kernel $K^R$. Precisely, we use the bound:
\begin{align}
\label{eq_schur}
\norm{(S-S^R) \hat{f}}_2 \leq
\max \left\{\esssup_\eta \int \abs{K^R(\eta,\xi)} d\xi,
\esssup_\xi \int \abs{K^R(\eta,\xi)} d\eta\right\} \norm{f}_2.
\end{align}

{\bf Step 1}. We show that
$\sup_\eta \int \abs{K^R(\eta,\xi)} d\xi \lesssim e^{-c'R^\alpha}$, for some constant $c'>0$. 

Let $\eta \in \hat{\bbR}^d$. By \eqref{eq_kernel1} we obtain
\begin{align*}
\int_{\hat{\bbR}^d} \abs{K^R(\eta,\xi)} d\xi
\lesssim \sum_{\omega \in \Omega \cap Y}
\int_{\hat{\bbR}^d} \chi_{\hat{\bbR}^d \setminus (Y + B_R)}(\eta)
e^{-c \abs{\omega-\eta}^\alpha} \cdot \chi_{V_\omega}(\xi) d\xi.
\end{align*}
Let us bound the integrand in the last expression.
When $\xi \in V_\omega$, $\abs{\xi-\omega} \leq \delta_{D^{\circ}}(\Omega) \leq L$. We can therefore
bound $e^{-c \abs{\omega-\eta}^\alpha} \lesssim
e^{-c \abs{\xi-\eta}^\alpha}$. Let $\omega \in \Omega \cap Y$. We can assume without loss of generality that 
$\eta \notin Y + B_R$. Consequently,
$\abs{\xi-\eta} \geq \abs{\eta-\omega} - \abs{\xi-\omega} \geq R-L$.
Hence,
\begin{align*}
&\int_{\hat{\bbR}^d} \abs{K^R(\eta,\xi)} d\xi
\lesssim \int_{\abs{\xi-\eta}\geq R-L} e^{-c \abs{\xi-\eta}^\alpha}
\sum_{\omega \in \Omega \cap Y} \chi_{V_\omega}(\xi) d\xi
\\
&\qquad \leq
\int_{\abs{\xi-\eta}\geq R-L} e^{-c \abs{\xi-\eta}^\alpha} d\xi
\leq e^{-c'(R-L)^\alpha} \lesssim e^{-c'R^\alpha},
\end{align*}
where $c'$ denotes a new constant.

{\bf Step 2}. We show that
$\sup_\xi \int \abs{K^R(\eta,\xi)} d\eta \lesssim e^{-c'R^\alpha}$, for some constant $c'>0$.

Let $\xi \in \hat{\bbR}^d$. By \eqref{eq_kernel1}, we may assume that
$\xi \in V_\omega$ for a (unique) $\omega \in \Omega \cap Y$, since otherwise
$K^R(\eta,\xi)=0$. Therefore,
\begin{align*}
&\int_{\hat{\bbR}^d} \abs{K^R(\eta,\xi)} d\eta
\leq \int_{\hat{\bbR}^d} e^{-c\abs{\omega-\eta}^\alpha}
\chi_{\hat{\bbR}^d \setminus (Y + B_R)}(\eta) d\eta.
\end{align*}
Since $\omega \in Y$, if $\eta \not\in Y + B_R$, it follows that
$\abs{\eta-\omega} \geq R$. As a consequence,
\begin{align*}
&\int_{\hat{\bbR}^d} \abs{K^R(\eta,\xi)} d\eta
\leq \int_{\abs{\eta-\omega}\geq R} e^{-c\abs{\omega-\eta}^\alpha} d\eta
\lesssim e^{-c'R^\alpha},
\end{align*}
for a new constant $c'$. This completes the proof.
\end{proof}

\subsection{Converse estimates}
We now derive converse estimates for the residual and the partial sum corresponding to nonuniform sampling. The 
following proposition is inspired by the work of Olevskii and Ulanovskii \cite{OlevskiiUlanovskii}
and follows closely their argument.
\begin{proposition}\label{prop_techncial}
Let $D\subseteq \bbR^d$ be a centered symmetric convex body. Let $\Omega\subseteq\hat{\bbR}^d$ be relatively 
separated
and $\varepsilon>0$ such that $4 \delta_{D^{\circ}}(\Omega)<1/(1+\varepsilon)$.
Let $\func \in \cL^1(D)$ such that $\int \func=1$.
Then there exists a constant $C_D>0$ -- that only depends on 
the smallest cube that contains $D$ --
such that for all measurable sets $Y \subseteq \hat{\bbR}^d$
\begin{align*}
\varepsilon^{d} \cos\left(2\pi(1+\varepsilon)  \delta_{D^\circ}(\Omega)\right)^2 \int_Y \left|\hat{f}(\xi)\right|^2 d\xi
&\leq 
\norm{\func}^2  \sum_{\omega \in \Omega \cap (Y+B_R)} \left|\hat{f}(\omega)\right|^2
\\
&+ 
C_D  n_\Omega^2 \norm{f}^2  \int_{\abs{\xi}>\varepsilon R} \left|\hat{\func}(\xi)\right|^2 d\xi,
\qquad
f \in\cL^2(D),\ R \geq 0,
\end{align*}
where 
$
n_\Omega = \sup_{x \in \hat{\bbR}^d} \# \left( \Omega \cap \left([-1/2,1/2]^d+x\right)\right)
$
is the covering number of $\Omega$.
\end{proposition}
\begin{proof}
Let $f \in\cL^2(D)$ and let $h(x) := \varepsilon^{-d}\func(x/\varepsilon)$.
For $\eta \in \hat{\bbR}^d$, let $g_\eta$ be defined by
\begin{align*}
\widehat{g_\eta}(\xi) := \hat{f}(\xi)\hat{h}(\xi-\eta).
\end{align*}
Hence, $g_\eta \in\cL^2((1+\varepsilon)D)$. Note that
$\delta_{((1+\varepsilon)D)^\circ}(\Omega)
=(1+\varepsilon) \delta_{D^\circ}(\Omega)$. Indeed, letting $r := (1+\varepsilon)$,
$(rD)^\circ = \tfrac{1}{r} D^\circ$, $\normo{\cdot}_{(rD)^\circ}=r \normo{\cdot}_{D^\circ}$
and consequently,
\begin{align*}
\delta_{(rD)^\circ}(\Omega)=\sup_{x\in\hat\bbR^d} \inf_{\omega\in\Omega} \normo{x-\omega}_{(rD)^\circ}
= r \sup_{x\in\hat\bbR^d} \inf_{\omega\in\Omega} \normo{x-\omega}_{D^\circ}
= r \delta_{D^\circ}(\Omega).
\end{align*}
Therefore, setting
\begin{align*}
K:= \cos(2\pi(1+\varepsilon) \delta_{D^\circ}(\Omega)),
\end{align*}
Proposition
\ref{prop_linf} implies that
\begin{align*}
K \left|\widehat{g_\eta}(\xi)\right| \leq
\sup_{\omega \in \Omega} \left|\hat{f}(\omega)\right|
\left|\hat{h}(\omega-\eta)\right|,
\qquad \xi \in \hat{\bbR}^d.
\end{align*}
In particular, since $\hat{h}(0)=\int h = 1$, 
$\hat{f}(\xi) = \widehat{g_\xi}(\xi)$, we obtain
\begin{align}
\label{eq_pr1}
K \left|\hat{f}(\xi)\right| \leq
\sup_{\omega \in \Omega} \left|\hat{f}(\omega)\right|
\left|\hat{h}(\omega-\xi)\right|,
\qquad \xi \in \hat{\bbR}^d.
\end{align}
Using this and setting $Y_{R} := Y+B_R$ we obtain 
\begin{align*}
&K^2 \int_Y \left|\hat{f}(\xi)\right|^2 d\xi
\leq \int_{Y} \sup_{\omega \in \Omega} \left|\hat{f}(\omega)\right|^2
 \left|\hat{h}(\omega-\xi)\right|^2 d\xi
\\
&\leq
\int_{Y} \sum_{\omega \in \Omega} \left|\hat{f}(\omega)\right|^2 \left|\hat{h}(\omega-\xi)\right|^2 d\xi
\\
&=
\int_{Y} \sum_{\omega \in \Omega \cap Y_{R}} \left|\hat{f}(\omega)\right|^2 \left|\hat{h}(\omega-\xi)\right|^2 d\xi
+
\int_{Y} \sum_{\omega \in \Omega \setminus Y_{R}} \left|\hat{f}(\omega)\right|^2 \left|\hat{h}(\omega-\xi)\right|^2 d\xi
\\
&\leq
\norm{h}^2 \sum_{\omega \in \Omega \cap Y_{R}} \left|\hat{f}(\omega)\right|^2
+ 
\sum_{\omega \in \Omega} \left|\hat{f}(\omega)\right|^2
\sup_{\omega \in \Omega\setminus Y_R} \int_Y \left|\hat{h}(\omega-\xi)\right|^2 d\xi
\\
&\leq
\varepsilon^{-d}  \norm{\func}^2 \sum_{\omega \in \Omega \cap Y_{R}} \left|\hat{f}(\omega)\right|^2
+ 
C_D  n_\Omega^2  \norm{f}^2 \int_{\abs{\xi}>R} \left|\hat{h}(\xi)\right|^2 d\xi
\\
&=
\varepsilon^{-d} \norm{\func}^2 \sum_{\omega \in \Omega \cap Y_{R}} \left|\hat{f}(\omega)\right|^2
+ 
C_D  n_\Omega^2  \norm{f}^2 \int_{\abs{\xi}>R} \left|\hat{\func}(\varepsilon \xi)\right|^2 d\xi
\\
&=
\varepsilon^{-d} \norm{\func}^2 \sum_{\omega \in \Omega \cap Y_{R}} \left|\hat{f}(\omega)\right|^2
+ 
\varepsilon^{-d}  C_D  n_\Omega^2  \norm{f}^2 \int_{\abs{\xi}>\varepsilon R} \left|\hat{\func}(\xi)\right|^2 
d\xi,
\end{align*}
which proves the claim.
\end{proof}
\begin{remark}
\label{rem_non_rel}
As the proof of Proposition \ref{prop_techncial} shows, the requirement that $\Omega$ be relatively separated can be 
dropped if $Y=\Rdst$ and $R \longrightarrow +\infty$.
\end{remark}

As an application of Proposition \ref{prop_techncial} we derive the following residual bound.
\begin{proposition}\label{weigt_sum_by_int}
Let $D\subseteq \bbR^d$ be a centered symmetric convex body. Let $\Omega\subseteq\hat{\bbR}^d$
be a closed countable set and let $\varepsilon>0$ be such that $(1+\varepsilon)\delta_{D^\circ}(\Omega)<1/4$.
Let $\alpha \in (0,1)$. Then for all measurable sets $Y \subseteq \hat{\bbR}^d$
\begin{align}
\label{eq_weigt_sum_by_int}
\int_Y \left|\hat{f}(\xi)\right|^2 d\xi
\leq 
C  \sum_{\omega \in (\Omega \cap (Y+B_R))} \mu_\omega \left|\hat{f}(\omega)\right|^2  
+ 
C e^{-cR^\alpha} \norm{f}^2,
\qquad
f \in\cL^2(D),\ R\geq 0,
\end{align}
where $c,C$ are positive constants that depend only on $D, \alpha$ and $\varepsilon$,
and $\mu_\omega$ are the Voronoi weights associated with the norm induced by $D^\circ$.
\end{proposition}
\begin{proof}
Let $\eta := \tfrac{\varepsilon}{2} \delta_{D^\circ}(\Omega)$, $\rho := \tfrac{\eta}{4}$,
and $\varepsilon'$ be defined by $(1+\varepsilon')(1+\tfrac{\varepsilon}{2}+\tfrac{\varepsilon}{8})=(1+\varepsilon)$.
Then $\rho < \tfrac{\eta}{2}$, $\varepsilon'>0$,
 and $(1+\varepsilon')\left(\delta_{D^\circ}(\Omega)+\eta+\rho\right)<1/4$.

Let $\func$ be a function such that $\supp(\func) \subseteq D$,
$\int \func=1$ and $\widehat{\func}(\xi) \lesssim  e^{-c' \abs{\xi}^\alpha}$ for some constant $c'>0$.
Let $f \in \cL^2(D)$. By Lemma \ref{lemma_add_weights}, there exists a set $\Omega'=\Omega'(f)$ that is $\eta-2\rho$ 
separated, has gap
\begin{align*}
\delta_{D^\circ}(\Omega') \leq \left(\delta_{D^\circ}(\Omega)+\eta+\rho\right) \leq \frac{1}{4(1+\varepsilon')}
\end{align*}
and satisfies \eqref{eq_lemma_add_weights}.

By Proposition \ref{prop_techncial} and the decay of $\widehat{\func}$ we conclude that
there exist positive constants $A, c, C$, that only depend on $\varepsilon>0$ and $D$ such that
\begin{align*}
A \int_Y \left|\hat{f}(\xi)\right|^2 d\xi
\leq C \sum_{\omega \in \Omega' \cap (Y+B_R)} \left|\hat{f}(\omega)\right|^2
+ 
C  {n^2_{\Omega'}} \norm{f}^2  e^{-c R^\alpha}.
\end{align*}
Note that since $\Omega'$ is $\eta-2\rho$ separated, $n_{\Omega'} \lesssim (\eta-2\rho)^{-d} < +\infty$.
We use \eqref{eq_lemma_add_weights} to conclude that
\begin{align*}
A \int_Y \left|\hat{f}(\xi)\right|^2 d\xi
\leq C \sum_{\omega \in \Omega' \cap (Y+B_{R+2\rho})} \mu_{\omega} \left|\hat{f}(\omega)\right|^2
+ 
C \norm{f}^2  e^{-c R^\alpha},
\end{align*}
for another constant $C>0$.
Since $\rho < \tfrac{1}{4}$, for $R \geq 1$, we can replace $R \mapsto R-2\rho$,
and absorb the corresponding change in the term $e^{-c R^\alpha}$ into the constants 
$c,C$ and obtain \eqref{eq_weigt_sum_by_int}. For $0 < R \leq 1$, \eqref{eq_weigt_sum_by_int} is 
trivially true.
\end{proof}
\begin{remark}
Note that the proof of Proposition \ref{weigt_sum_by_int} uses a certain sampling set $\Omega'$ that depends on the 
function $f$ being sampled. The explicit estimate in Proportion \ref{prop_techncial} shows how the lower frame bound 
associated to $\Omega'$ depends on the geometry of $\Omega'$, and thus allows us to get an estimate independent of $f$. 
This kind of reasoning would not be available without quantitative information on the lower frame bound.
\end{remark}

\subsection{Converse estimates at critical density}
We now derive a version of the previous estimates for the case of critical sampling density. As expected, the error 
decay in much milder than in the oversampling case.
\begin{proposition}\label{p:converse_est}
The following estimates hold for $f \in \cL^2([-1/2,1/2]^d)$:
\begin{align}
\label{eq_conv_1}
&\int_{B_M} \abs{\hat{f}(\xi)}^2 d\xi
\lesssim
\sum_{k \in \bbZ^d \cap B_{R}}\abs{\hat{f}(k)}^2
+ \frac{M}{R} \norm{f}^2,
\qquad
R, M \geq 0,
\\
\label{eq_conv_2}
&\int_{\hat{\Rst}^d\setminus B_M} \abs{\hat{f}(\xi)}^2 d\xi
\lesssim
\sum_{k \in \bbZ^d \setminus B_{R}}\abs{\hat{f}(k)}^2
+ \frac{R}{M} \norm{f}^2,
\qquad
R, M \geq 0,
\end{align}
where the implied constant depends only on the dimension $d$.
\end{proposition}

\begin{proof}
{\bf Step 1.} \emph{We prove \eqref{eq_conv_1} in dimension $d=1$.}
Let $\sinc(x) = \frac{\sin(\pi x)}{\pi x}$. Hence, $\widehat{\sinc}=\chi_{[-1/2,1/2]}$.
Note first that the estimate is trivial if $R \leq 2M$. Assume that
$R \geq 2M$ and let us consider the operators $S_M, S^R_M: \ell^2(\Zst) \to \cL^2(\hRdst)$,
\begin{align*}
(S_M a)(\xi) &:= \sum_{k \in \Zst} a_k \cdot \sinc(\xi-k) \cdot \chi_{ B_{M}}(\xi),
\\
(S^R_M a)(\xi) &:= \sum_{k \in \Zst \cap B_{R}} a_k \cdot \sinc(\xi-k) \cdot \chi_{ B_{M}}(\xi).
\end{align*}
For $f \in \cL^2([-1/2,1/2])$, we let $a_k := \hat{f}(k)$ and estimate
\begin{align*}
&\left(\int_{B_M} \abs{\hat{f}(\xi)}^2 d\xi\right)^{1/2}
= \norm{S_M a}
\leq \norm{S^R_M a} + \norm{(S_M-S^R_M)a}
\\
&\qquad=
\Norm{\sum_{k \in \Zst \cap B_{R}} \hat{f}(k) \sinc(\cdot-k)}_{\cL^2(B_{M})} + \norm{(S_M-S^R_M)a}
\\
&\qquad
\leq \Norm{\sum_{k \in \Zst \cap B_{R}} \hat{f}(k) \sinc(\cdot-k)}_{\cL^2(\hRst)} +
\norm{S_M-S^R_M}_{\ell^2\to \cL^2} \norm{a}
\\
&\qquad = 
\left(\sum_{k \in \Zst \cap B_{R}} \abs{\hat{f}(k)}^2  \right)^{1/2} + 
\norm{S_M-S^R_M}_{\ell^2\to \cL^2} \norm{f}.
\end{align*}
Therefore,
\begin{align*}
\int_{B_M} \abs{\hat{f}(\xi)}^2 d\xi
\lesssim 
\sum_{k \in \Zst \cap B_{R}} \abs{\hat{f}(k)}^2+ \norm{S_M-S^R_M}^2_{\ell^2\to \cL^2}\norm{f}^2,
\end{align*}
and it suffices to show that $\norm{S_M-S^R_M}_{\ell^2\to \cL^2} \lesssim \sqrt{\frac{M}{R}}$. To this end, note
that for $\xi \in B_{M}$ and $k \notin \Zst \cap B_{R}$,
$\abs{\xi} \leq M \leq R/2 \leq \tfrac{1}{2} \abs{k}$. As a consequence,
\begin{align*}
\abs{\xi-k} \geq \abs{k} - \abs{\xi} \geq \tfrac{1}{2} \abs{k},
\end{align*}
and therefore,
\begin{align*}
\abs{\sinc(\xi-k)} \lesssim \frac{1}{k}.
\end{align*}
Hence, for $a \in \ell^2(\Zst)$ we bound
\begin{align*}
&\abs{(S_M-S^R_M)a(\xi)} = 
\abs{\sum_{k \in \Zst, \abs{k} > R} a_k \cdot \sinc(\xi-k) \cdot \chi_{ B_{M}}(\xi)}
\\
&\qquad \lesssim
\chi_{ B_{M}}(\xi) 
\sum_{k \in \Zst: \abs{k} > R} \frac{\abs{a_k}}{\abs{k}}.
\end{align*}
By Cauchy--Schwarz,
\begin{align*}
|(S_M - S^R_M)a(\xi) | &\leq 
\chi_{B_M}(\xi) \| a \|_2 \left(\sum_{k \in \Zst: \abs{k} > R} \abs{k}^{-2} \right)^{1/2}
\\
&\lesssim \chi_{B_M}(\xi) \| a \|_2  R^{-1/2}.
\end{align*}
Hence
\begin{align*}
\| (S_M - S^R_M)a \|_2 & \lesssim \| a \|_2 \sqrt{M/R},
\end{align*}
as required.

{\bf Step 2.} \emph{We extend \eqref{eq_conv_1} to $d > 1$.}
We proceed by induction, with the notation $\xi=(\xi_1,\xi_*) \in 
\hat{\Rst}\times\hat{\Rst}^{d-1}$.
Without loss of generality we use the infinity-norm - which we still denote by $\abs{\cdot}$
to keep the notation simple. Applying the result in dimension $d-1$ 
to the function obtained by taking a partial Fourier transform of $f$ in the first variable, we obtain
\begin{align*}
\int_{\abs{\xi_*} \leq M} \abs{\hat{f}(\xi_1,\xi_*)}^2 d\xi_*
\lesssim \sum_{k_* \in \bbZ^{d-1}, \abs{k_*} \leq R} 
\abs{\hat{f}(\xi_1,k_*)}^2
+ \frac{CM}{R} \int_{\hat{\Rst}^{d-1}} \abs{\hat{f}(\xi_1,\xi_*)}^2 d\xi_*,
\qquad \xi_1 \in \hat{\Rst}.
\end{align*}
We now integrate on $\xi_1$ and apply the one-dimensional version of the result:
\begin{align*}
&\int_{\abs{\xi} \leq M }\abs{\hat{f}(\xi)}^2 d\xi
=
\int_{\abs{\xi_1} \leq M} 
\int_{\abs{\xi_*} \leq M} \abs{\hat{f}(\xi_1,\xi_*)}^2 d\xi_* d\xi_1
\\
&\qquad \lesssim \sum_{k_* \in \bbZ^{d-1}, \abs{k_*} \leq R} 
\int_{\abs{\xi_1} \leq M} \abs{\hat{f}(\xi_1,k_*)}^2 d\xi_1
+ \frac{M}{R} \norm{f}^2
\\
&\qquad
\lesssim \sum_{k_* \in \bbZ^{d-1}, \abs{k_*} \leq R}
\left( 
\sum_{k_1 \in \bbZ, \abs{k_1} \leq R} \abs{\hat{f}(k_1,k_*)}^2
+
\frac{M}{R} \int_{\hat{\Rst}} \abs{\hat{f}(\xi_1,k_*)}^2 d\xi_1
\right)
+ \frac{M}{R} \norm{f}^2
\\
&\qquad = \sum_{k \in \bbZ^d, \abs{k} \leq R} \abs{\hat{f}(k_1,k_*)}^2
+
\frac{M}{R} \int_{\hat{\Rst}} \sum_{k_* \in \bbZ^{d-1}, \abs{k_*} \leq R} 
\abs{\hat{f}(\xi_1,k_*)}^2 d\xi_1
+ \frac{M}{R} \norm{f}^2
\\
&\qquad \lesssim \sum_{k \in \bbZ^d, \abs{k} \leq R} \abs{\hat{f}(k)}^2
+ \frac{M}{R} \norm{f}^2.
\end{align*}
This completes the proof.

{\bf Step 3.} \emph{We prove \eqref{eq_conv_2} in dimension $d=1$}. As before, we assume without loss of generality
that $M \geq 2R$ and consider the operators $T_M, T^R_M: \ell^2(\Zst) \to \cL^2(\hRdst)$,
\begin{align*}
(T_M a)(\xi) &:= \sum_{k \in \Zst} a_k \cdot \sinc(\xi-k) \cdot \chi_{\hat{\Rst}\setminus B_{M}}(\xi),
\\
(T^R_M a)(\xi) &:= \sum_{k \in \Zst, \abs{k}>R} a_k \cdot \sinc(\xi-k)
\cdot \chi_{\hat{\Rst}\setminus B_{M}}(\xi).
\end{align*}
With $a_k := \hat{f}(k)$ we estimate
\begin{align*}
&\left(\int_{\abs{\xi}>M} \abs{\hat{f}(\xi)}^2 d\xi\right)^{1/2}
= \norm{T_M a} 
\leq \norm{T^R_M a} + \norm{T_M - T^R_M} \norm{f}
\\
&\leq \left(\sum_{k \in \Zst, \abs{k}>R} \abs{\hat{f}(k)}^2\right)^{1/2} \norm{T_M - T^R_M} \norm{f},
\end{align*}
so it suffices to show that $\norm{T_M - T^R_M} \lesssim \sqrt{R/M}$. To this end, note that
if $\xi \notin B_{M}$ and $\abs{k} \leq R$, then $\abs{\xi} \geq M \geq 2R \geq 2 \abs{k}$, so
$\abs{\xi-k} \geq \abs{\xi} - \abs{k} \geq \tfrac{1}{2} \abs{\xi}$. Consequently,
\begin{align*}
&\abs{(T_M - T^R_M)a(\xi)} \leq 
\sum_{k \in \Zst, \abs{k} \leq R} \abs{a_k} \cdot \abs{\sinc(\xi-k)}
\cdot \chi_{\hat{\Rst}\setminus B_{M}}(\xi)
\\
&\qquad
\lesssim
\sum_{k \in \Zst, \abs{k} \leq R} \abs{a_k} \abs{\xi}^{-1} \cdot \chi_{\hat{\Rst}\setminus B_{M}}(\xi)
\lesssim
\norm{a} \sqrt{R} \abs{\xi}^{-1} \cdot \chi_{\hat{\Rst}\setminus B_{M}}(\xi),
\end{align*}
and
\begin{align*}
\norm{(T_M - T^R_M)a}^2_2 \lesssim \norm{a}^2_2 R \int_{\abs{\xi}>M} \abs{\xi}^{-2} d\xi
\lesssim \norm{a}^2_2 \frac{R}{M},
\end{align*}
as desired.

{\bf Step 4.} \emph{We extend \eqref{eq_conv_2} to $d > 1$.} We use again the notation $\xi=(\xi_1,\xi_*) \in 
\hat{\Rst}\times\hat{\Rst}^{d-1}$ and the infinity-norm. Applying the result in dimension $1$ 
to the function obtained by taking a partial Fourier transform of $f$ in the last variables, we obtain
\begin{align*}
\int_{\abs{\xi_1} > M} \abs{\hat{f}(\xi_1,\xi_*)}^2 d\xi_1
\lesssim
\sum_{k_1 \in \Zst, \abs{k_1}>R}
\abs{\hat{f}(k_1,\xi_*)}^2 +  \frac{R}{M} \int_{\hat{\Rst}} \abs{\hat{f}(\xi_1,\xi_*)}^2 d\xi_1.
\end{align*}
We integrate on $\xi_*$ and use Parseval's identity to conclude that
\begin{align*}
&\int_{\abs{\xi_1} > M} \abs{\hat{f}(\xi)}^2 d\xi
\lesssim
\sum_{k_1 \in \Zst, \abs{k_1}>R}
\int_{\hat{\Rst}^{d-1}}
\abs{\hat{f}(k_1,\xi_*)}^2 d\xi_* +  \frac{R}{M} \norm{f}^2
\\
&\qquad 
= \sum_{k \in \Zst, \abs{k_1}>R}
\abs{\hat{f}(k)}^2 +  \frac{R}{M} \norm{f}^2
\\
&\qquad \leq
\sum_{k \in \Zst, \abs{k}>R}
\abs{\hat{f}(k)}^2 +  \frac{R}{M} \norm{f}^2.
\end{align*}
Finally, using a similar estimate with the other coordinates $\xi_k$ instead of $\xi_1$ we conclude that
\begin{align*}
&\int_{\abs{\xi} > M} \abs{\hat{f}(\xi)}^2 d\xi
\leq \sum_{k=1}^d  
\int_{\abs{\xi_k} > M} \abs{\hat{f}(\xi)}^2 d\xi
\lesssim
\sum_{k \in \Zst, \abs{k}>R}
\abs{\hat{f}(k)}^2 +  \frac{R}{M} \norm{f}^2.
\end{align*}
\end{proof}

\subsection{Explicit estimates for lower frame bounds under sharp density conditions}\label{ss:explicit}

Using Proposition \ref{prop_techncial} we now derive the following results.

\begin{corollary}[Lower frame bound for unweighted exponentials]
\label{coro_expl}
Let $D \subseteq \bbR^d$ be a centered symmetric convex body and let
$\Omega \subseteq \hat{\bbR}^d$ be a closed countable set such that $\delta_{D^\circ}(\Omega)<1/4$. For an 
$\varepsilon>0$ such that 
$(1+\varepsilon) \delta_{D^\circ}(\Omega)<1/4$, we have
\begin{align*}
{\mes{D} \varepsilon^{d} \cos\left(2\pi(1+\varepsilon) \delta_{D^\circ}\right)^2}
\norm{f}^2 \leq \sum_{\omega\in\Omega}\left|\hat f(\omega)\right|^2,
\qquad f \in\cL^2(D).
\end{align*}
\end{corollary}
\begin{proof}
We apply Proposition \ref{prop_techncial} with $Y=\hat{\bbR}^d$,
$\func := \mes{D}^{-1} \chi_{D}$ and let $R \rightarrow  +\infty$. Note that in this case we do not need $\Omega$ to be 
relatively separated (cf. Remark \ref{rem_non_rel}).
\end{proof}

\begin{corollary}[Lower frame bound for weighted exponentials]
\label{coro_expl_2}
Let $D \subseteq \bbR^d$ be a centered symmetric convex body and let
$\Omega \subseteq \hat{\bbR}^d$ be a closed countable set such that $\delta_{D^\circ}=\delta_{D^\circ}(\Omega)<1/4$. 
For 
$\varepsilon,\eta>0$ such that $(1+\varepsilon)\left(\delta_{D^\circ}+\eta\right)<1/4$, we have
\begin{align*}
\mes{D^\circ} \mes{D} \left(\frac{\eta\varepsilon}{6}\right)^d 
\cos\left(2\pi(1+\varepsilon)\left(\delta_{D^\circ}+\eta\right)\right)^2 \norm{f}^2\leq\sum_{\omega\in\Omega} \mu_\omega\left|\hat 
f(\omega)\right|^2,
\qquad f \in\cL^2(D).
\end{align*}
\end{corollary}
\begin{proof}
We use Lemma \ref{lemma_add_weights} with $Y=\hat{\bbR}^d$, $\eta' := \frac{2}{3}\eta$ and $0<\rho<\eta'/2$ to obtain a set $\bar{\Omega}$ satisfying
\eqref{eq_lemma_add_weights}. Since $\delta_{D^\circ}(\bar\Omega)
\leq \delta_{D^\circ}(\Omega) + \eta' + \rho < \delta_{D^\circ}(\Omega) + \eta$,
we can apply Corollary \ref{coro_expl} to $\bar\Omega$. 
We combine the conclusion of Corollary \ref{coro_expl} with 
\eqref{eq_lemma_add_weights} and let $\rho \rightarrow \eta'/2$ to obtain the desired estimate.
\end{proof}

\section{The remaining proofs}
\label{sec_gran_finale}
We finally prove the results announced in the Introduction.
\begin{proof}[Proof of Theorem \ref{th_samp}]
The theorem follows immediately from Corollary \ref{coro_expl_2} by making the explicit choice:
$\varepsilon := \left((4\delta_{D^{\circ}})^{-1/2} - 1\right)  ( 1- \tfrac{1}{d+2})$
and $\eta :=\varepsilon \delta_{D^{\circ}}$. Note that this choice is admissible since
$(1+\varepsilon)\left(\delta_{D^\circ}+\eta\right)=
(1+\varepsilon)^2 \delta_{D^\circ} < 1/4$.
\end{proof}
\begin{remark}
The values of $\epsilon$ and $\eta$ have been chosen to asymptotically optimize the lower bound in
Corollary \ref{coro_expl_2} in the two regimes $\delta_{D^\circ}(\Omega) \longrightarrow 1/4$
and $d \longrightarrow +\infty$.

\end{remark}
\begin{proof}[Proofs of Theorems \ref{th_1} and \ref{th_2}]
The estimates in \eqref{eq_V1} and \eqref{eq_Vstar1} follow from Proposition \ref{prop_residual}
by taking $Y=B_K$ and $Y=\Rdst \setminus B_K$ respectively. Similarly, the estimates in \eqref{eq_V2} and 
\eqref{eq_Vstar2} follow from Proposition \ref{weigt_sum_by_int} with $Y=B_K$ and $Y=\Rdst \setminus B_K$. Finally, 
Theorem \ref{th_2} follows immediately from
Proposition \ref{p:converse_est}.
\end{proof}

\begin{proof}[Proof of Corollary \ref{coro_univ_rate}]
Let $\varepsilon>0$ and let $A=A_\varepsilon$ be the bound in \eqref{eq_frame_bound} associated
with any set $\Omega$ with gap $\delta_{D^\circ}(\Omega) \leq (1-\varepsilon)1/4$. 
Theorem \ref{th_1}
implies that $V_*(\cR_N,\Omega, 2c N^\alpha) \lesssim V_*(\cR_N,c N^\alpha) + C e^{-c' \sqrt{c N^{\alpha}}}$, 
for some constant $c'>0$. Hence, we can 
choose $c>0$ such that $V_*(\cR_N,\Omega, 2c N^\alpha) \leq \tfrac{1}{2}A$. 
For $f \in \cR_N$ with $\norm{f}=1$, we 
simply estimate
\begin{align*}
\sum_{w \in \Omega \cap B_{2cN^\alpha}} \mu_{\omega} \abs{\hat{f}(\omega)}^2
\geq 
\sum_{w \in \Omega} \mu_{\omega} \abs{\hat{f}(\omega)}^2 - V_*(\cR,\Omega, 2c N^\alpha)
\geq \tfrac{1}{2}A,
\end{align*}
and the conclusion follows.
\end{proof}

\begin{proof}[Proof of Corollary \ref{coro_transfer}]
We only prove part (b); part (a) can be proved similarly.
Let $\theta>0$. By Theorem \ref{th_2}, there exists a constant $C>0$ such that $V_*(\cR_N, K) \leq C V_*(\cR_N, 
\mathbb{Z}^d, M) + C \frac{M}{K}$. By hypothesis there exists $c'_\theta>0$ such that $\sup_N V_*(\cR_N, \mathbb{Z}^d, 
c'_\theta N^\alpha) \leq \tfrac{\theta}{2C}$. For $c_\theta >0$, we conclude that
\begin{align*}
V_*(\cR_N, c_\theta N^\alpha) \leq
C V_*(\cR_N, \mathbb{Z}^d, c'_\theta N^\alpha) + C \frac{c'_\theta}{c_\theta}
\leq \frac{\theta}{2} + C \frac{c'_\theta}{c_\theta}.
\end{align*}
Hence, it suffices to let $c_\theta =  
\frac{2 C c'_\theta}{\theta}$.
\end{proof}

\begin{proof}[Proof of Corollary \ref{coro_barrier}]
We treat the case $D=[-\tfrac{1}{2},\tfrac{1}{2}]^d$ and $\Omega_0=\Zst^d$; the other case is analogous.
Let $c,\gamma,A>0$ be such that $V(\cR_N, \Omega, c N^\gamma) \geq A$. We want to show that $\gamma \geq \alpha$.
Suppose on the contrary that $\gamma < \alpha$ and let $\beta \in (\gamma,\alpha)$. By Theorem \ref{th_1}, it follows that
\begin{align*}
A \leq V(\cR_N, \Omega, K_N) \lesssim V(\cR_N, K_N + M) + e^{-c'\sqrt{M}},
\end{align*}
for some constant $c'>0$. Hence, choosing $M \gg 1$, we conclude that
\begin{align*}
A \lesssim V(\cR_N, K_N + M).
\end{align*}
Having fixed $M$, we now use Theorem \ref{th_2} to obtain
\begin{align*}
A \lesssim V(\cR_N, K_N + M) \lesssim 
V(\cR_N, \Zst^d, N^\beta) + \frac{K_N+M}{N^\beta}
\leq V(\cR_N, \Zst^d, N^\beta) + O(N^{\gamma-\beta}).
\end{align*}
Hence, for $N \gg 0$, $V(\cR_N, \Zst^d, N^\beta) \gtrsim A$. 
Therefore, if we let $\widetilde{K}_N := c N^\beta$, with $c \gg 1$,
it follows that $\inf_N V(\cR_N, \Zst^d, \widetilde{K}_N) >0$.
Since $\beta<\alpha$, this is contradicts
the assumption that $\alpha$ is a stability barrier for the sampling problem associated with $\mathbb{Z}^d$.
\end{proof}

\end{document}